\documentclass[12pt]{report}
\usepackage[english]{babel}
\usepackage{blindtext}
\usepackage{url}
\usepackage[utf8]{inputenc}
\usepackage{amsmath}
\usepackage{subfigure}
\usepackage{fancyhdr}
\usepackage{vmargin}
\usepackage{amsfonts}
\usepackage{latexsym}
\usepackage{amsthm}
\usepackage{graphicx}
\usepackage{multirow}
\usepackage{ulem}
\usepackage{pgf,tikz}
\usetikzlibrary{arrows}
\usepackage{nicematrix}
\usepackage{cleveref}
\usepackage{parskip}
\usepackage{semantic}
\usepackage{indentfirst}
\usepackage{nomencl}
\usepackage{setspace}
\usepackage{datetime}

\makenomenclature

\usepackage{etoolbox}
\renewcommand\nomgroup[1]{%
  \vspace{20pt} \item[\bfseries
  \ifstrequal{#1}{A}{Set notation}{%
  \ifstrequal{#1}{B}{Pattern notation}{%
  \ifstrequal{#1}{C}{4-tuple-letter notation}{%
  \ifstrequal{#1}{D}{Other notation}{}}}}%
]
}

\newtheorem{Lem}{Lemma}
\newtheorem{Cor}{Corollary}
\newtheorem{The}{Theorem}

\theoremstyle{definition}
\newtheorem{Exam}{Example}

\newtheorem{Rem}{Remark}
\newtheorem{Def}{Definition}

\newcommand*{\dprime}{^{\prime\prime}\mkern-1.2mu}

\newcommand*\numbAtBack[1]{\hspace*{0em plus 1fill}\makebox{(#1)}}

\newdate{date}{29}{01}{2022}

\def\v{\mathtt{v}}
\def\d{\mathtt{d}}
\def\p{\mathtt{p}}
\def\r{\mathtt{r}}
\def\l{\mathtt{l}}

\def\Sset{\mathcal{S}}
\def\close{\mathrm{closer}}
\def\open{\mathrm{opener}}
\def\out{\mathrm{outsider}}
\def\insi{\mathrm{insider}}
\def\red{\mathtt{red}}
\def\Destop{\mathtt{Destop}}
\def\Desbot{\mathtt{Desbot}}
\def\eob{\mathtt{-eob}}
\def\inv{\mathtt{inv}}
\def\stat{\mathtt{stat}}
\def\maj{\mathtt{maj}}
\def\mak{\mathtt{mak}}
\def\makl{\mathtt{makl}}
\def\mad{\mathtt{mad}}
\def\madl{\mathtt{madl}}

\def\des{\mathtt{des}}

\def\Dbot{\mathtt{Dbot}}%
\def\Dtop{\mathtt{Dtop}}
\def\DTOP{\mathtt{DTOP}}
\def\DBOT{\mathtt{DBOT}}
\def\Ddif{\mathtt{Ddif}}
\def\Res{\mathtt{Res}}
\def\Les{\mathtt{Les}}

\usepackage[
    backend=biber,
    style=alphabetic,
  ]{biblatex}
\usepackage{graphicx}
\graphicspath{{images/}}

\usepackage{enumitem}
\setlistdepth{9}
\setlist[itemize,1]{label=$\bullet$}
\setlist[itemize,2]{label=$\bullet$}
\setlist[itemize,3]{label=$\bullet$}
\setlist[itemize,4]{label=$\bullet$}
\setlist[itemize,5]{label=$\bullet$}
\setlist[itemize,6]{label=$\bullet$}
\setlist[itemize,7]{label=$\bullet$}
\setlist[itemize,8]{label=$\bullet$}
\setlist[itemize,9]{label=$\bullet$}
\renewlist{itemize}{itemize}{9}

\setmarginsrb{3 cm}{3 cm}{2 cm}{3 cm}{1 cm}{1.5 cm}{1 cm}{1.5 cm}

\title{BACHELOR THESIS}								
\author{Ta Thi Phuong Lien}								
\date{January 29, 2022}											

\makeatletter
\let\thetitle\@title
\let\theauthor\@author
\let\thedate\@date
\makeatother

\begin{document}


\begin{titlepage}
    \begin{figure}[ht]
	   \minipage{0.6\textwidth}
			\includegraphics[height=2.2cm]{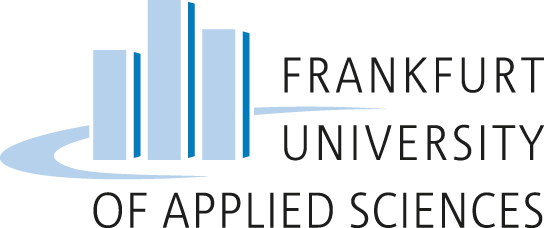}
	   \endminipage
	   \minipage{0.4\textwidth}
	        \vspace*{0.25 cm}
			\includegraphics[height=2.2cm]{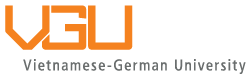}
		\endminipage
	\end{figure}
	\begin{center}
	    {\Large \bfseries Vietnamese German University}\\
	    {\large Department of Computer Science}\\
	    \vspace*{0.5 cm}
        {\large \bfseries Frankfurt University of Applied Sciences}\\
        {\large Faculty of Computer Science and Engineering}\\
        \rule{4cm}{0.4pt}
	\end{center}
	\centering
    \vspace*{1.5 cm}
    {\LARGE \bfseries \thetitle}\\
    \vspace*{1 cm}
	{ \Large \bfseries Mahonian Statistics and Vincular Patterns}\\
	\vspace*{0.3 cm}
	{ \Large \bfseries on Permutations over Multisets}\\
	\vspace*{1.5 cm}
	\begin{center}
	    \begin{tabular}{  m{6cm}  m{6cm} } 
            {\large \bfseries Student name:} & {\large Ta Thi Phuong Lien}\\
            {\large \bfseries Matriculation number:} & {\large 1276082}\\
            {\large \bfseries Supervisor:} & {\large Dr. Tran Thi Thu Huong}\\
            {\large \bfseries Co-Supervisor:} & {\large Dr. Huynh Trung Hieu}
        \end{tabular}
	\end{center}
	\vspace*{1 cm}
	{\large \thedate}\\[2 cm]
 
	\vfill
	
	\thispagestyle{empty}
	
\end{titlepage}


\maketitle

\tableofcontents
\pagebreak

\addcontentsline{toc}{chapter}{Disclaimer}
\chapter*{Disclaimer}
I hereby declare that the information reported in the paper is the result of my own, original, individual work, except where references are made. This thesis is finished under the guidance and supervision of Dr. Tran Thi Thu Huong and Dr. Huynh Trung Hieu in Vietnamese – German University. I also certify that this undergraduate dissertation has not been previously or concurrently submitted for other degrees or other universities institutions.

\vspace*{2cm}
Ta Thi Phuong Lien
\newpage

\addcontentsline{toc}{chapter}{Acknowledgment}
\chapter*{Acknowledgment}
I would like to send my sincere thanks to my supervisor, Dr. Tran Thi Thu Huong, for her patience and support throughout this thesis study. Without her assistance, I would have not been able to finish this thesis.
\newpage

\addcontentsline{toc}{chapter}{List of figures}
\listoffigures

\addcontentsline{toc}{chapter}{Notation}
\mbox{}
\nomenclature[A,01]{\(\Sset_n\)}{The set of all permutation of length $n$}
\nomenclature[A,02]{\(\Sset_(A,m)\)}{The set of all permutation over the multiset $(A,m)$}
\nomenclature[D,01]{\(\mid A \mid\)}{The number of elements in a set $A$}
\nomenclature[B,01]{\(\red(\pi)\)}{The reduced form of the permutation $\pi$}
\nomenclature[A,03]{\(L_A\)}{The set of 4-tuple-letter set of permutations in set $A$}
\nomenclature[D,02]{\(\overline{x,y}\)}{Natural numbers from $x$ to $y$}
\nomenclature[D,03]{\(\pi_i\)}{The $i^{th}$ entry of the permutation (either on set or on multiset) $\pi$}
\nomenclature[C,06]{\(W_{(v,d)}\)}{The 4-tuple-letter that has the value is $v$ and the duplicate index is $d$}
\nomenclature[A,04]{\(W\vert_{\v = v}\)}{The set of 4-tuple-letters in the set $W$ that have their values are $v$}
\nomenclature[B,02]{\((pattern)\vert_{v_1 = v_2}\)}{The number of occurrence of pattern $pattern$ such that the letter $v_1$ in $pattern$ is corresponding to the letter $v_2$ in occurrence}
\nomenclature[C,04]{\(\r(A)\)}{The right embracing number of $A$. $A$ is either letter in permutation or 4-tuple-letter}
\nomenclature[C,05]{\(\l(A)\)}{The left embracing number of $A$. $A$ is either letter in permutation or 4-tuple-letter}
\nomenclature[C,03]{\(\p(A)\)}{The position of $A$. $A$ is either letter in permutation or 4-tuple-letter}
\nomenclature[C,02]{\(\d(A)\)}{The duplicate index of $A$. $A$ is either a letter in permutation or a 4-tuple-letter}
\nomenclature[C,01]{\(\v(A)\)}{The value of $A$. $A$ is either a letter in permutation or a 4-tuple-letter}
\nomenclature[D,04]{\(O(B)\)}{The opener of descent block $B$}
\nomenclature[D,05]{\(C(B)\)}{The closer of descent block $B$}
\nomenclature[D,06]{\(v\eob\)}{$v$-embraceable open block}
\nomenclature[D,07]{\(\# v\eob(S)\)}{The number of $v$-embraceable open block of insertion $S$}
\nomenclature[D,08]{\(\#\)}{The number of}

\begin{singlespace}
\printnomenclature[1.2in]
\end{singlespace}

\addcontentsline{toc}{chapter}{Abstract}
\chapter*{Abstract}
Most Mahonian statistics can be expressed as a linear combination of vincular patterns. This is not only true with statistics on the permutation set, but it can also be applied for statistics on the permutation with repetition set \cite{vincular}. By following the method extending the vincular patterns combinations presented by Kitaev and Vajnovszki \cite{vincular}, we discover 6-vincular-patterns combinations of $\mad$ and $\madl$ extensions that are possible to be Mahonian. Some of these have been proved to be Mahonian on repetitive permutations by Clarke, Steingr{\'i}msson and Zeng  \cite{CLARKE1997237}, while the rest are new statistics extensions.

In this thesis, we determine combinations of vincular pattern extension of $\mad$ and $\madl$ in Clarke, Steingr{\'i}msson and Zeng's paper, which have been proved to be Mahonian on the repetitive permutations. This result will be used to support the proof of Mahonity of the new statistics extensions. We show that these new statistics extensions are also Mahonian by constructing an involution $\Phi$ on repetitive permutations, which preserves the descents statistics and transforms new statistics extensions to Mahonian $\mad$ and $\madl$ extensions of \cite{CLARKE1997237}. \begin{flushright}
\textit{Keywords: vincular pattern, Mahonian, $\mad$, $\madl$}
\end{flushright}

\newpage

\addcontentsline{toc}{chapter}{Introduction}
\chapter*{Introduction}
Various statistics on permutations have been discussed in the past \cite{Babson2000}\cite{CLARKE1996}\cite{Denert}\cite{mahonWord}. The number of inversions is the simplest and best-known Mahonian statistic. The inversions statistics is the number of pairs of indices $(i, j)$ such that $i < j$ and $\pi_i > \pi_j$. MacMahon defined the major index ($\maj$) of a permutation, which has the same distribution as $\inv$, in which the major index statistics is defined as the sum of the indices $i$ such that $\pi_i > \pi_{i+1}$ \cite{mahonWord}. Any permutation statistics with the same distribution as $\maj$ is called Mahonian.

Since then, many new Mahonian statistics have been identified. In 1996, Clarke, Steingr{\'i}msson and Zeng \cite{CLARKE1996} have defined or re-defined new Mahonian permutation statistics $\mad$, $\mak$, $\madl$, $\makl$. These statistics are also proved to be Mahonian not only on permutations but also on words \cite{CLARKE1997237}, more precisely, these statistics are Mahonian on the sets of rearrangements of the letters of a given word $w$. In other words, these four statistics are Mahonian on the permutations over the multisets of letters.

Most Mahonian statistics on permutations can be expressed as a linear combination of vincular patterns \cite{Babson2000}. The notion of a vincular pattern can be extended to words \cite{vincular}. In \cite{CLARKE1997237}, the authors point out two $\mad$ statistics extensions, namely $\mad_1$ and $\mad_2$, and two $\madl$ statistics extensions, namely $\madl_1$ and $\madl_2$, that are Mahonian on repetitive permutations. In this paper, we represent these statistics extensions as combinations of vincular patterns.

By following the method extending the vincular patterns combinations mentioned in \cite{vincular} and running experimental tests, we have found eight extensions of $\mad$ and $\madl$ represented as vincular pattern combinations that are promising to be Mahonian on repetitive permutations. These found statistics extensions are:
\begin{itemize}
    \item four statistics extensions of $\mad$ and $\madl$ defined by Clarke, Steingr{\'i}mssonn and Zeng \cite{CLARKE1997237}
    \item two new $\mad$ extensions, namely $\mad_3$ and $\mad_4$
    \item two new $\madl$ extensions, namely $\madl_3$ and $\madl_4$
\end{itemize}
For proving 4 new statistics extensions are Mahonian, we construct an involution $\Phi$ on repetitive permutations which preserves the descents statistics and transforms $\mad_3$, $\mad_4$, $\madl_3$, $\madl_4$ to $\mad_1$, $\mad_2$, $\madl_1$, $\madl_2$, respectively.

This thesis consists of 3 main chapters:
\begin{itemize}
    \item Chapter 1 introduces Mahonian statistics $\mad$, $\madl$ on permutation and related terms
    \item Chapter 2 represents statistics in Clark's paper \cite{CLARKE1997237} as vincular patterns extensions $\mad_1$, $\mad_2$, $\madl_1$, and $\madl_2$. At the end of this chapter, we conjectures new Mahonian vincular patterns extensions $\mad_3$, $\mad_4$, $\madl_3$, and $\madl_4$.
    \item Chapter 3 shows a constructive bijection $\Phi$ on repetitive permutation such that:
    \begin{align*}
        (\des, \mad_3)w &= (\des, \mad_1) \Phi(w)\\
        (\des, \mad_4)w &= (\des, \mad_2) \Phi(w)\\
        (\des, \madl_3)w &= (\des, \madl_1) \Phi(w)\\
        (\des, \madl_4)w &= (\des, \madl_2) \Phi(w)
    \end{align*}
\end{itemize}

\newpage

\chapter{Mahonian statistics on permutation set}
\section{Mahonian statistics}
Before introducing the Mahonian statistics, we define some related definitions about major index and inversion.
\begin{Def}
An \textit{inversion} in a permutation $\pi$ is a pair of indices $(i, j)$ such that $i < j$ and $\pi_i > \pi_j$.\\
\textit{Inversion statistics}, denoted by $\inv$, is the number of inversions.
\end{Def}
\begin{Def}
A \textit{descent} in a permutation $\pi$ is an index $i$ such that $\pi_i > \pi_{i+1}$.\\
\textit{Major index}, denoted by $\maj$, is the sum of the indices of the descents.
\end{Def}
\begin{Def}
\textit{Mahonian} statistics is the \textit{statistics} that has the same \textit{distribution} as \textit{major index} statistics or \textit{inversion} statistics.
\end{Def}
Besides major index, many other Mahonian statistics have been identified. In this paper, we focus on two statistics $\mad$ and $\madl$, which have been proved to be Mahonian on permutations and permutations over mutiset. Follows are the definition of $\mad$, $\madl$ on permutation, and other relating statistics.
\begin{Def}
Let $i$ be a descent in a permutation $\pi$. The number $\pi_i$ is a \textit{descent top}, and $\pi_{i+1}$ is a \textit{descent bottom}. The \textit{descent bottoms sum}, denoted by $\Dbot$, of $\pi$ is the sum of the descent bottoms of $\pi$. Similarly, the \textit{descent tops sum}, denoted by $\Dtop$, of $\pi$ is the sum of the descent tops of $\pi$. The \textit{descent difference}, denoted by $\Ddif$, of $\pi$ is the difference of descent tops sum and descent bottoms sum.
\end{Def}
\begin{Def}
Let $\pi = \pi_1\pi_2\dots \pi_n$ be a permutation.\\
The \textit{(right) embracing numbers} of $\pi$ are the numbers $\r(\pi_1),\r(\pi_2),\dots,\r(\pi_n)$ where $\r(\pi_i)$ is the number of descent blocks in $\pi$ that are strictly to the right of $\pi_i$ and that embrace $\pi_i$.\\
The \textit{right embracing sum} of $\pi$, denoted by$\Res(\pi)$, is defined by
\begin{align*}
    \Res(\pi) = \r(\pi_1) + \r(\pi_2) + \dots + \r(\pi_n)
\end{align*}
By replacing "right" by "left" in the above definition, we define \textit{left embracing numbers}, \textit{left embracing sum} $\Les$ in an analogous way, 
\end{Def}

The definition of $\mad$ and $\madl$ was first introduced in \cite{CLARKE1996}
\begin{Def}
The $\mad$ statistics of permutation $\pi$ was defined as the sum of the right embracing sum and the descent difference in $\pi$.
\begin{align*}
    \mad(\pi) = \Res(\pi) + \Dbot(\pi)
\end{align*}
\end{Def}
\begin{Def}
The $\madl$ statistics of permutation $\pi$ was defined as the sum of the left embracing sum and the descent difference in $\pi$.
\begin{align*}
    \madl(\pi) = \Les(\pi) + \Dbot(\pi)
\end{align*}
\end{Def}
\section{Vincular Patterns}
In \cite{Babson2000}, Babson and Steingr{\'i}msson pointed out that many Mahonian statistics can be expressed as linear combinations of statistics counting occurrences of \textit{vincular patterns}. Follows are definitions related to vincular pattern.
\begin{Def}
The \textit{reduced form} of a permutation $\pi = \pi_1\pi_2\dots \pi_r$, denoted by $\red(\pi)$, obtained by replacing the $i^{th}$ smallest letter of $\pi$ by $i$, for $i = 1, 2, \dots, r$.
\end{Def}
\begin{Exam}
$\red(8516) = 4213$
\end{Exam}
\begin{Def}
A permutation $\pi = \pi_1\pi_2\dots \pi_n$ has an \textit{occurrence} of the \textit{classical pattern} $\tau \in \Sset_r$ if there exist $1 \leq i_1 < i_2 < · · · < i_r \leq n$ such that $\tau = \red(\pi_{i_1} \pi_{i_2}\dots \pi_{i_r})$.
\end{Def}
\begin{Exam}
The permutation $341562$ has seven occurrences of the pattern $231$, which are the subsequences $341$, $342$, $352$, $362$, $452$, $462$, and $562$.
\end{Exam}
\begin{Def}
A \textit{vincular pattern} of length $r$ is a generalization of the notion of a classical pattern that allows to present the adjacencies. It is a pair $(\tau; X)$ such that $\tau$ is a permutation in $\Sset_r$ and $X \subseteq \{0,1,\dots,r\}$ is a set of adjacencies. 

In this paper, we only consider about vincular patterns with at most 1 adjacency and the adjacency is neither $0$ nor $r$. If $X = \emptyset$, the vincular pattern is classic. If $X = \{j\}$ such that $j \in \overline{1,r-1}$, adjacency is indicated in pattern presence by underlining $\tau_j$ and $\tau_{j+1}$. A permutation $\pi = \pi_1\pi_2\dots \pi_n$ has an occurrence of the vincular pattern $(\tau,\{j\}$ in which $j \in \overline{1,r-1}$ if
\begin{itemize}
    \item there exist $1 \leq i_1 < i_2 < · · · < i_r \leq n$ such that $\tau = \red(\pi_{i_1} \pi_{i_2}\dots \pi_{i_r})$.
    \item $i_{j+1} = i_j + 1$
\end{itemize}
\end{Def}
\begin{Exam}
The permutation $341562$ has four occurrences of the pattern $2\underline{31}$, which are the subsequences $341$, $362$, $462$, and $562$.
\end{Exam}

\section{Natural extensions of statistics to words}\label{sec13}
\begin{Def}\label{defEx}
An \textit{extension of a vincular pattern} $(\tau,X)$ is the combination of the original vincular pattern with some of its "weaker" vincular patterns, which are obtained by replacing an entry in $\tau$ by its weak counterpart. More formally describe, let $W = \{( red(\tau_1\tau_2\dots(\tau_i-1)\tau_{i+1}\dots\tau_r) , X):i\in [r]\}$ be a set of "weaker" vincular patterns of $(\tau,X)$, in which $r$ is the length of $\tau$. An extension of $(\tau,X)$ is obtained by combining $(\tau,X)$ and some elements in $W$. An extension of a linear combination of vincular patterns $(\tau_1,X_1)+ (\tau_2,X_2) + \dots + (\tau_s,X_s)$ is the statistics obtained by extending some of patterns $(\tau,X)$.
\end{Def}
Many statistics on permutations can be expressed as a combination of vincular patterns. Some of these are Mahonian not only on permutations but also on words such as $\inv$, $\maj$ \cite{mahonWord}, $\stat$ \cite{vincular}, $\mad$, $\mak$, $\madl$, $\makl$ \cite{CLARKE1997237}, etc. With any statistics $S$ satisfying these two conditions, the vincular patterns representation of $S$'s statistics extension on words is an extension of $S$'s vincular patterns representation. Follows are vincular patterns combinations of some Mahonian statistics and its Mahonian extensions:
\begin{enumerate}
    \item A vincular patterns extension of $\inv$ on permutation $\pi$ is $$   \inv(\pi) = (\underline{23}1 + \underline{31}2 + \underline{32}1 + \underline{21}) \pi$$, and its standard vincular patterns extension on word defined by MacMahon \cite{mahonWord} which was proved to be Mahonian is $$ \inv(w) = (\underline{23}1 + \underline{31}2 + \underline{32}1 + \underline{21} + \underline{22}1 + \underline{21}1)w $$, in which
    \begin{itemize}
        \item $\underline{23}1$ extended as $\underline{23}1 + \underline{22}1$
        \item $\underline{31}2$ extended as $\underline{31}2 + \underline{21}1$
    \end{itemize}
    \item A vincular patterns extension of $\maj$ on permutation $\pi$ is $$   \maj(\pi) = (1\underline{32} + 2\underline{31} + 3\underline{21} + \underline{21}) \pi$$, and its standard vincular patterns extension on word defined by MacMahon \cite{mahonWord} which was proved to be Mahonian is $$ \maj(w) = (1\underline{32} + 2\underline{31} + 3\underline{21} + \underline{21} + 1\underline{21} + 2\underline{21})w $$, in which
    \begin{itemize}
        \item $1\underline{32}$ extended as $1\underline{32} + 1\underline{21}$
        \item $2\underline{31}$ extended as $2\underline{31} + 2\underline{21}$
    \end{itemize}
    \item In \cite{Babson2000}, the authors identified vincular patterns of $\mad$ and $\madl$ statistics on permutation:
    \begin{align*}
        \mad(w) &= (2. 2\underline{31}+ \underline{31}2 +\underline{21})w\\
        \madl(w) &= (\underline{31}2 + 2\underline{31} + \underline{31}2 + \underline{21})w\\
        \mak(w) &= (1\underline{32} + \underline{32}1 + \underline{21} + 2\underline{31})(w)\\
        \makl(w) &= (\underline{31}2 + \underline{32}1 + \underline{21} + 1\underline{32})(w)
    \end{align*}
    In the next chapter, we will define vincular patterns combinations of $\mak$, $\makl$, $\mad$ and $\madl$ extensions that were proved to be Mahonian in \cite{CLARKE1997237}
\end{enumerate}

\chapter{Statistics $\mad$ and $\madl$ on repetitive permutations set}\label{chap1}
\section{Statistics extensions of $\mad$ and $\madl$ by \cite{CLARKE1997237}}
Given a word $w$, in \cite{CLARKE1997237}, the authors studied statistics on the class of rearrangement of the word $R(w)$. These statistics of the word $w$ are defined as follows:
\begin{enumerate}
    \item $\mak_l(w)=\Dbot_l(w) + \Res(w)$
    \item $\mak_r(w)=\Dbot_r(w) + \Res(w)$
    \item $\mad_l(w) = \Ddif_l(w) + \Res(w)$
    \item $\mad_r(w)=\Ddif_r(w) + \Res(w)$
    \item $\makl_l(w)=\Dbot_l(w)+\Les(w)$
    \item $\makl_r(w)=\Dbot_r(w)+\Les(w)$
    \item $\madl_l(w)=\Ddif_l(w) + \Les(w)$
    \item $\madl_r(w) = \Ddif_r(w) + \Les(w)$
\end{enumerate}
in which:
\begin{enumerate}
    \item \textit{Height value} of $w_i$, denoted by $h(w_i)$, is the number of letters in $w$ that are strictly smaller than $w_i$ plus 1
    \item \textit{Left value} of $w_i$, denoted by $v_l(w_i)$, is the sum of $h(w_i)$ and $l(w_i)$, in which $l(w_i)$ is the number of letter to the left of $w_i$ and equal to $w_i$.
    \item \textit{Right value} of $w_i$, denoted by $v_r(w_i)$, is the sum of $h(w_i)$ and $r(w_i)$, in which $r(w_i)$ is the number of letter to the right of $w_i$ and equal to $w_i$.
    
    \item The \textit{(right) embracing numbers} of $w$ are the numbers $\r(w_i)$ where $e_i$ is the number of descent blocks in $w$ that are strictly to the right of $w_i$ and that embrace $w_i$.\\
    If $v$ is a letter of $w$, we say that a is embraced by B if $C(B) \geq a > O(B)$.
    
    The \textit{right embracing sum} of $w$ ($\Res$) is the sum of embracing numbers of $w$.
    \item The definitions about \textit{left embracing numbers} $\l(w_i)$ and \textit{left embracing sum} ($\Les$) can be defined by replacing "right" by "left" in the above definitions.
    \item \textit{Descent tops sum} of $w$, denoted by $\Dtop(w)$ is the sum of the height of descent tops of $w$.
    \item \textit{Descent bottoms sum}s of $w$, denoted by $\Dbot_l(w)$ and $\Dbot_r(w)$ are the sum of the left values and the sum of the right values of descent bottoms of $w$.
    \item \textit{Descent difference}s of $w$, denoted by $\Ddif_l(w)$ and $\Ddif_r(w)$ are defined by
    \begin{align*}
        \Ddif_l(w) = \Dtop(w) - \Dbot_l(w)\\
        \Ddif_r(w) = \Dtop(w) - \Dbot_r(w)
    \end{align*}
\end{enumerate}
\section{Representations of extended statistics in term of vincular patterns}
Next, we show that we can represent the statistics mentioned above as a linear combination of vincular patterns.
\begin{Lem}\label{lma30}
Follows are the vincular patterns extensions of statistics in \cite{CLARKE1997237}:
\begin{enumerate}
\item $\Res(w)= (2\underline{31} + 2\underline{21})(w)$; 
\item $\Les(w) = (\underline{31} 2 + \underline{21}2) (w)$ 
\item $\sum_{w_i\in \DBOT (w)} l(w_i)= 1\underline{21} (w)$; 
\item $\sum_{w_i\in \DBOT (w)} r(w_i)= \underline{21}1 (w)$; 
\item $\sum_{w_i\in \DBOT (w)} h(w_i) = (1\underline{32} + \underline{32}1 + \underline{21})(w)$; 
\item $\Dtop(w)=\sum_{w_i\in \DTOP (w)} h(w_i) = (1\underline{32} + \underline{32}1 + \underline{31}2 + 2\underline{31} + 1\underline{21} + \underline{21}1 + 2.\underline{21})(w)$\\
\item Consequently,
\begin{align*}
    \Dbot_l(w) =& (1\underline{21}+1\underline{32}+\underline{32}1+\underline{21})(w)\\
    \Ddif_l(w) =& (\underline{31}2 + 2\underline{31} + \underline{21} + \underline{21}1) (w)\\
    \Dbot_r(w) =&  (1\underline{32} + \underline{32}1 + \underline{21} + \underline{21}1)(w)\\
    \Ddif_r(w) =& (\underline{31}2 + 2\underline{31} + \underline{21} + 1\underline{21})(w)\\
    \mak_l(w) =& (1\underline{21} + 1\underline{32} + \underline{32}1 + \underline{21} + 2\underline{31} + 2\underline{21})(w)\\
    \mad_l(w) =& (\underline{21}1 + \underline{31}2 + \underline{21} + 2\underline{31} + 2\underline{31} + 2\underline{21})(w)\\
    \makl_l(w) =& (1\underline{21} + 1\underline{32} + \underline{32}1 +\underline{21} + \underline{31} 2 +\underline{21}2)(w) \\
    \madl_l(w) =& (\underline{31}2 + 2\underline{31} + \underline{21} + \underline{21}1 + \underline{31} 2 +\underline{21}2)(w)\\
    \mak_r(w) =& (1\underline{32} + \underline{32}1 + \underline{21} + \underline{21}1 + 2\underline{31} + 2\underline{21})(w) \\
    \mad_r(w) =& (\underline{31}2 + 2\underline{31} + \underline{21} + 1\underline{21} + 2\underline{31} + 2\underline{21})(w)\\
    \makl_r(w) =& (1\underline{32} + \underline{32}1 + \underline{21} + \underline{21}1 + \underline{31} 2 + \underline{21}2)(w)\\
    \madl_r(w) =& (\underline{31}2 + 2\underline{31} + \underline{21} + 1\underline{21} + \underline{31} 2 + \underline{21}2)(w)
\end{align*}
\end{enumerate}
\end{Lem}
\begin{proof}
\begin{enumerate}
\item $\Res(w)= (2\underline{31} + 2\underline{21})(w)$;
\begin{multline*}
    \r(w_i) = \mid\{(i,j,j+1): j > i, w_{j} > w_i > w_{j+1}, j=\overline{1,n-1}\}\mid\\
    + \mid\{(i,j,j+1): j > i, w_{j} = w_i > w_{j+1},j=\overline{1,n-1}\}\mid
\end{multline*}
\begin{align*}
    Res(w) &= \Sigma_{\text{i from 1 to n}}\r(w_i)\\
    &=\Sigma_{\text{i from 1 to n}}\mid\{(i,j,j+1): j > i, w_{j} > w_i > w_{j+1},j=\overline{1,n-1}\}\mid \\
    & \qquad \qquad + \Sigma_{\text{i from 1 to n}}\mid\{(i,j,j+1): j > i, w_{j} = w_i > w_{j+1},j=\overline{1,n-1}\}\mid\\
    &= (2\underline{31} + 2\underline{21})(w)
\end{align*}
\item $\Les (w) = (\underline{31} 2 +\underline{21}) (w)$; 
\begin{multline*}
    \l(w_i) = \mid\{(j,j+1,i): j + 1 < i, w_{j} > w_i > w_{j+1}, j=\overline{1,n-1}\}\mid \\
    + \mid\{(j,j+1,i): j+1<i, w_{j} = w_i > w_{j+1},j=\overline{1,n-1}\}\mid
\end{multline*}
\begin{align*}
    Res(w) &= \Sigma_{\text{i from 1 to n}}\r(w_i)\\
    &=\Sigma_{\text{i from 1 to n}}\mid\{(j,j+1,i): j+1 < i, w_{j} > w_i > w_{j+1},j=\overline{1,n-1}\}\mid \\
    & \qquad \qquad + \Sigma_{\text{i from 1 to n}}\mid\{(j,j+1,i): j+1 < i, w_{j} = w_i > w_{j+1},j=\overline{1,n-1}\}\mid\\
    &= (\underline{31}2 + \underline{21}2)(w)
\end{align*}
\item $\sum_{w_i\in \DBOT (w)} l(w_i)= 1\underline{21} (w)$
\begin{align*}
    l(w_i) &= \mid \{\text{letter to the left of $w_i$ and equal to $w_i$}\}\mid \\
    &= \mid \{(j,i): j < i, w_j = w_i, j=\overline{1,n}\} \mid
\end{align*}
\begin{align*}
    \Sigma_{w_i\in \DBOT}l(w_i) &= \Sigma_{w_i\in \DBOT}\mid \{(j,i): j < i, w_j = w_i, j=\overline{1,n}\} \mid\\
    &= \mid \{(j,i-1,i): j < i-1, w_j = w_i < w_{i-1}, j=\overline{1,n}, i=\overline{1,n}\} \mid\\
    &= (1\underline{21})(w)
\end{align*}
\item $\sum_{w_i\in \DBOT (w)} r(w_i)= \underline{21}1 (w)$
\begin{align*}
    r(w_i) &= \mid \{\text{letter to the right of $w_i$ and equal to $w_i$}\}\mid \\
    &= \mid \{(i,j): j > i, w_i = w_j, j=\overline{1,n}\} \mid
\end{align*}
\begin{align*}
    \Sigma_{w_i\in \DBOT}l(w_i) &= \Sigma_{w_i\in \DBOT}\mid \{(i,j): i < j, w_i = w_j, j=\overline{1,n}\} \mid\\
    &= \mid \{(i-1,i,j): i < j, w_j = w_i < w_{i-1}, j=\overline{1,n}, i=\overline{1,n}\} \mid\\
    &= (\underline{21}1)(w)
\end{align*}
\item $\sum_{w_i\in \DBOT (w)} h(w_i) = (1\underline{32}+ \underline{32}1+\underline{21})(w)$;
\begin{align*}
    h(w_i) &= \mid\{\text{(letters in $w$ that are strictly smaller than $w_i$)}\}\mid + 1\\
    &= \mid \{(i,j): w_i > w_j, j=\overline{1,n}\} \mid + 1\\
    &= \mid \{(i,j): w_i > w_j, j=\overline{1,n}\} \mid + \mid \{(i)\} \mid
\end{align*}
\begin{align*}
    \sum_{w_i\in \DBOT (w)} h(w_i) &= \sum_{w_i\in \DBOT (w)} (\mid \{(i,j): w_i > w_j, j=\overline{1,n}\} \mid + \mid \{(i)\}\mid)\\
    &= \mid \{(i-1,i,j): w_{i-1} > w_i > w_j, i=\overline{1,n}, j=\overline{1,n}\} \mid \\
    & \qquad \qquad + \mid \{(i-1,i): w_{i-1} > w_i, i=\overline{1,n}\}\mid\\
    &= \mid \{(i-1,i,j): i < j, w_{i-1} > w_i > w_j, i=\overline{1,n}, j=\overline{1,n}\} \mid \\
    & \qquad \qquad + \mid \{(i-1,i,j): i-1 > j, w_{i-1} > w_i > w_j, i=\overline{1,n}, j=\overline{1,n}\} \mid\\
    & \qquad \qquad + \mid \{(i-1,i): w_{i-1} > w_i, i=\overline{1,n}\}\mid\\
    &= (\underline{32}1 + 1\underline{32} + \underline{21})(w)
\end{align*}
\item $\Dtop (w)=\sum_{w_i\in \DTOP (w)} h(w_i) = (1\underline{32}+ \underline{32}1+\underline{31}2+2\underline{31}+2.\underline{21})(w)+1\underline{21}+\underline{21}1)(w)$;
\begin{align*}
    \sum_{w_i\in \DTOP (w)} h(w_i) &= \sum_{w_i\in \DTOP (w)} (\mid \{(i,j): w_i > w_j, j=\overline{1,n}\} \mid + \mid \{(i)\}\mid)\\
    &= \mid \{(i,i+1,j): w_i > w_{i+1}, w_i > w_j, i=\overline{1,n}, j=\overline{1,n}\} \mid \\
    & \qquad \qquad + \mid \{(i,i+1): w_i > w_{i+1}, i=\overline{1,n}\}\mid\\
    &= \mid \{(i,i+1,j): w_i > w_j > w_{i+1}, i=\overline{1,n}, j=\overline{1,n}\} \mid \\
    & \qquad \qquad + \mid \{(i,i+1,j): w_i > w_{i+1} > w_j, i=\overline{1,n}, j=\overline{1,n}\} \mid \\
    & \qquad \qquad + \mid \{(i,i+1,j): i+1 = j, w_i > w_{i+1} = w_j, i=\overline{1,n}, j=\overline{1,n}\} \mid \\
    & \qquad \qquad + \mid \{(i,i+1,j): i+1\neq j, w_i > w_{i+1} = w_j, i=\overline{1,n}, j=\overline{1,n}\} \mid \\
    & \qquad \qquad + \mid \{(i-1,i): w_{i-1} > w_i, i=\overline{1,n}\}\mid\\
    &= \mid \{(i,i+1,j): j > i+1, w_i > w_j > w_{i+1}, i=\overline{1,n}, j=\overline{1,n}\} \mid \\
    & \qquad \qquad +  \mid \{(i,i+1,j): j < i, w_i > w_j > w_{i+1}, i=\overline{1,n}, j=\overline{1,n}\} \mid \\
    & \qquad \qquad + \mid \{(i,i+1,j): j > i+1, w_i > w_{i+1} > w_j, i=\overline{1,n}, j=\overline{1,n}\} \mid \\
    & \qquad \qquad + \mid \{(i,i+1,j): j < i, w_i > w_{i+1} > w_j, i=\overline{1,n}, j=\overline{1,n}\} \mid \\
    & \qquad \qquad + \mid \{(i,i+1,j): i+1 = j, w_i > w_{i+1} = w_j, i=\overline{1,n}, j=\overline{1,n}\} \mid \\
    & \qquad \qquad + \mid \{(i,i+1,j): w_i > w_{i+1} = w_j, i=\overline{1,n}, j=\overline{1,n}\} \mid \\
    & \qquad \qquad + \mid \{(i-1,i): j > i+1, w_{i-1} > w_i, i=\overline{1,n}\}\mid\\
    & \qquad \qquad + \mid \{(i-1,i): j < i, w_{i-1} > w_i, i=\overline{1,n}\}\mid\\
    &= (\underline{31}2 + 2\underline{31} + \underline{32}1 + 1\underline{32} + \underline{21} + \underline{21}2 + 2\underline{21} +  \underline{21})(w)
\end{align*}
\item Consequently,
$$\Dbot_l(w) = \Sigma_{w_i\in \DBOT}(l(w_i) + h(w_i)) = ((1\underline{21}) + (1\underline{32}) + (\underline{32}1) + (\underline{21}) )(w)$$
$$\Ddif_l(w) = \Dtop(w) - \Dbot_l(w) = (\underline{21}1) + (\underline{31}2) + (\underline{21}) + (2\underline{31})$$
$$\Dbot_r(w) = \sum_{w_i\in\Dbot(w)}(h(w_i) + r(w_i)) = (1\underline{32} + \underline{32}1 + \underline{21} + \underline{21}1)(w)$$
$$\Ddif_r(w) = \Dtop(w) - \Dbot_r(w) = \underline{31}2 + 2\underline{31} + \underline{21} + 1\underline{21}$$

$$\mak_l(w) = \Dbot_l(w) + \Res(w) = (1\underline{21} + 1\underline{32} + \underline{32}1 + \underline{21} + 2\underline{31} + 2\underline{21})(w)$$
$$\mad_l(w) = \Ddif_l(w) + \Res(w) = (\underline{21}1 + \underline{31}2 + \underline{21} + 2\underline{31} + 2\underline{31} + 2\underline{21})(w)$$
$$\makl_l(w) = \Dbot_l(w) + \Les(w) = (1\underline{21} + 1\underline{32} + \underline{32}1 +\underline{21} + \underline{31} 2 +\underline{21}2)(w) $$
$$\madl_l(w) = \Ddif_l(w) + \Les(w) = (\underline{31}2 + 2\underline{31} + \underline{21} + \underline{21}1 + \underline{31} 2 +\underline{21}2)(w)$$

$$\mak_r(w) = \Dbot_r(w) + \Res(w) =(1\underline{32} + \underline{32}1 + \underline{21} + \underline{21}1 + 2\underline{31} + 2\underline{21})(w) $$
$$\mad_r(w) = \Ddif_r(w) + \Res(w) =(\underline{31}2 + 2\underline{31} + \underline{21} + 1\underline{21} + 2\underline{31} + 2\underline{21})(w)$$
$$\makl_r(w) = \Dbot_r(w) + \Les(w) = (1\underline{32} + \underline{32}1 + \underline{21} + \underline{21}1 + \underline{31} 2 + \underline{21}2)(w)$$
$$\madl_r(w) = \Ddif_r(w) + \Les(w) = (\underline{31}2 + 2\underline{31} + \underline{21} + 1\underline{21} + \underline{31} 2 + \underline{21}2)(w)$$
\end{enumerate}
\end{proof}
\section{Other statistics extensions}
Follow is a method we use to obtain the list of vincular patterns combinations that have Mahonity on words:
\begin{itemize}
    \item Follows the definition \cref{defEx}, from a given combination of vincular patterns, we can generate a list of all possible 6-vincular-patterns combinations.
    \item By doing the experimental test with all combinations in that list, we eliminate combinations that cannot be Mahonian on words.
\end{itemize}
Applying above method for vincular patterns combinations of $\mak$, $\makl$, $\mad$ and $\madl$ mentioned in \cref{sec13} gives us the two first column of \cref{tab1}. Some records in \cref{tab1} was proved to be Mahonian on repetitive permutations set by Clarke, Steingr{\'i}mssonn and Zeng. It seems like their works, so far, are the only statistics extensions on repetitive permutations for $\mad$ and $\madl$. The statistics extensions $\mad_3$, $\mad_4$, $\madl_3$, and $\madl_4$ have not been defined elsewhere. 
\begin{table}[]
    \centering
    \begin{tabular}{|c|c|c|}
\hline
Statistic& Extensions & Notes\\
\hline
$\mak$     & $\mak_1 = (1\underline{21} + 1\underline{32} + \underline{32}1 + \underline{21} + 2\underline{31} + 2\underline{21})$ & Clarke, Steingr{\'i}mssonn and Zeng\\
\hline
$\mak$     & $\mak_2 = (\underline{21}1 + 1\underline{32} + \underline{32}1 + \underline{21} + 2\underline{31} + 2\underline{21})$ & Clarke, Steingr{\'i}mssonn and Zeng\\
\hline
$\mad$ & $\mad_1 = 2. 2\underline{31}+ \underline{31}2 +\underline{21}1+2\underline{21}+\underline{21}$& Clarke, Steingr{\'i}mssonn and Zeng\\
\hline 
$\mad$ & $\mad_2 = 2. 2\underline{31}+ \underline{31}2 +1\underline{21}+2\underline{21}+\underline{21}$ & Clarke, Steingr{\'i}mssonn and Zeng\\
\hline 
$\mad$ & $\mad_3 = 2. 2\underline{31}+ \underline{31}2 +\underline{21}2+\underline{21}1+\underline{21}$ & new\\
\hline 
$\mad$ & $\mad_4 = 2. 2\underline{31}+ \underline{31}2 +\underline{21}2+1\underline{21}+\underline{21}$ & new\\
\hline 
$\makl$ & $\makl_1 = \underline{31}2+ \underline{32}1 +1\underline{32}+\underline{21}1+\underline{21}2+\underline{21}$ & Clarke, Steingr{\'i}mssonn and Zeng\\
\hline 
$\makl$ & $\makl_2 = \underline{31}2+ \underline{32}1 +1\underline{32}+1\underline{21}+\underline{21}2+\underline{21}$ & Clarke, Steingr{\'i}mssonn and Zeng\\
\hline
$\madl$ & $\madl_1 = \underline{31}2 + 2\underline{31} + \underline{31}2 + \underline{21} + 1\underline{21} + \underline{21}2$ & Clarke, Steingr{\'i}mssonn and Zeng\\
\hline
$\madl$ & $\madl_2 = \underline{31}2 + 2\underline{31} + \underline{31}2 + \underline{21} + \underline{21}2 + \underline{21}1$ & Clarke, Steingr{\'i}mssonn and Zeng\\
\hline 
$\madl$ & $\madl_3 = \underline{31}2 + 2\underline{31} + \underline{31}2 + \underline{21} + 1\underline{21} + 2\underline{21}$ & new\\
\hline
$\madl$ & $\madl_4 = \underline{31}2 + 2\underline{31} + \underline{31}2 + \underline{21} + 2\underline{21} + \underline{21}1$ & new\\
\hline
\end{tabular}
\caption{Statistic extensions of $\mad$, $\madl$, $\mak$, and $\makl$}
\label{tab1}
\end{table}

The following chapters will prove that the new statistics extensions are Mahonian.

\chapter{Another vincular patterns extensions of $\mad$ and $\madl$}
In this chapter, we construct and prove an involution $\Phi$ such that:
\begin{align*}
    \mad_3 w &= \mad_1 \Phi(w)\\
    \mad_4 w &= \mad_2 \Phi(w)\\
    \madl_3 w &= \madl_1 \Phi(w)\\
    \madl_4 w &= \madl_2 \Phi(w)
\end{align*}
In \cite{CLARKE1997237}, the authors studied $\mad_1$, $\mad_2$, $\madl_1$, and $\madl_2$ on the rearrangement class of the word $R(w)$, that is the set of all words that can be obtained by permuting the letters of $w$, in which $w$ is the word on the totally ordered alphabet $A$. Since we map $\mad_3$, $\mad_4$, $\madl_3$, and $\madl_4$ to their works,  the scope for studying the new $\mad$ and $\madl$ extensions will be also on $R(w)$.

We consider a multiset $M = (A,m)$ such that with every order letter $a$ in $A$, the multiplicity of $a$ in $M$ equals to the number of occurrences of $a$ in $w$. It is clear that $w$ is a permutation on the multiset $M$, and $R(w)$ is a set of permutations over the multiset $M$. Studying statistics on permutations over the multiset $M$ equals studying statistics on the rearrangement class of the word $R(w)$.

Bellow is the formal definition of the multiset.
\begin{Def}
A \textit{multiset} is formally defined as a 2-tuple $(A, m)$ where 
\begin{itemize}
    \item $A$ is the \textit{underlying set} of the multiset, formed from its distinct elements
    \item ${\displaystyle m\colon A\to \mathbb {Z} ^{+}}$ is a function from $A$ to the set of the positive integers, giving the \textit{multiplicity}, that is, the number of occurrences, of the element $a$ in the multiset\\
    The multiplicity of $a$ in the multiset denote by $m(a)$
\end{itemize}
\end{Def}
We denote the word $\Sset_{(A,m)}$ or $\Sset_{(A,m)}$ is the set of all permutations over $M=(A,m)$ multiset. It is clear that $\Sset_{(A,m)} = \Sset_{(A,m)} = R(w)$. 

This chapter contains three sections. The first section defines all the terms and annotations used throughput the chapter. The second section presents the function $\Phi$ and shows it is a bijection. The last section shows the relation between the occurrences of the vincular pattern counting of $w$ and the $\Phi(w)$ one.

\pagebreak
\section{Basic concepts and notations}
This section covers definitions that will be used throughput this chapter. Some terms in this chapter are defined differently to their definitions in \cref{chap1}.
\begin{Def}\label{def0}
Give a word $w=w_1w_2\dots w_n$ in $\Sset_{(A,m)}$. The \textit{value} of $w_i$, denoted by $\v(w_i)$, is the letter in $A$ presenting for $w_i$
\end{Def}

\begin{Def}\label{def1}
Give a word $w=w_1w_2\dots w_n$ in $\Sset_{(A,m)}$. A \textit{descent} in $w$ is a triple $(i,w_i,w_{i+1})$ such that $i \in [n-1]$ and $w_i > w_{i+1}$. $w_i$ is called \textit{descent top}, and $w_{i+1}$ is called \textit{descent bottom}.

The number of occurrences of descent in word $w$ is counted by $(\underline{21})w$.
\end{Def}
\begin{Def}\label{def3}
Given a word $w = w_1w_2\dots w_n$, we separate $w$ into its descent blocks by putting in dashes between $w_i$ and $w_{i+1}$ whenever $w_i \leq w_{i+1}$. A \textit{descent block} is a maximal continuous sub-word of $w$ which lies between two dashes. A descent block is an \textit{outsider} if it has only one letter; otherwise, it is a proper descent block. The leftmost letter of a proper descent block is its \textit{closer} and the rightmost letter is its \textit{opener}. A letter which lies strictly inside a descent block is an \textit{insider}.

The position of $w_i$ is denoted by $\p(w_i)$. The opener and the closer of block $B$ is denoted by $O(B)$ and $C(B)$, respectively.
\end{Def}
\begin{Rem}\label{rem9}
It is clear that insiders are descent tops and descent bottoms; outsiders are non-descent-tops and non-descent bottoms; openers are descent bottoms and non-descent tops; closers are descent tops and non-descent bottoms.
\end{Rem}

\begin{Def}\label{def2}
The \textit{embraced} definition here is different to the one in \cite{CLARKE1997237}. Let $B$ be a proper descent block of $w$. If $a$ is a letter of $w$, we say that $a$ is embraced by $B$ if $O(B) < a < C(B)$.

The \textit{right embracing number} of the word $w=w_1w_2\dots w_n$ are the numbers $\r(w_1)$, $\r(w_2)$, $\dots$, $\r(w_n)$, where $\r(w_i)$ is the number of descent blocks in $w$ that are strictly to the right of $w_i$ and that embrace $w_i$. The \textit{right embracing sum} of $w$ is the sum of all right embracing numbers of $w$.

Similarly, the \textit{left embracing number} of the word $w=w_1w_2\dots w_n$ are the numbers $\l(w_1)$, $\l(w_2)$, $\dots$, $\l(w_n)$, where $\l(w_i)$ is the number of descent blocks in $w$ that are strictly to the left of $w_i$ and that embrace $w_i$. The \textit{left embracing sum} of $w$ is the sum of all left embracing numbers of $w$

The right embracing sum and the left embracing sum of $w$ word $w$ is counted by $(2\underline{31})w$ and $(\underline{31}2)w$, respectively.
\end{Def}
\begin{Exam}
Given $w=3-61-7-84$. The letter $w_4$ has:
\begin{itemize}
    \item value: $\v(w_4) = 7$
    \item position: $\p(w_4) = $outsider
    \item right embracing number: $\r(w_4) = 1$
    \item left embracing number: $\l(w_4) = 0$
\end{itemize}
\end{Exam}

\begin{Rem}\label{rem4}
From the way we define right embracing number and left embracing number, it is clear that
\begin{align}
    \Res(w) = (2\underline{31})w\\
    \Les(w) = (\underline{31}2)w
\end{align}
\end{Rem}

\begin{Def}\label{index}
Given a word $w=w_1w_2\dots w_n$. With letters that have the same value, we shall enumerate them from left to right. These enumerated indices are called \textit{duplicate index}. The duplicate index of $w_i$ is denoted by $\d(w_i)$.
\end{Def}

\begin{Def}
We define a new annotation $((\tau,X))\vert _{\tau_i = v}$. Given a word $w$ and the vincular pattern $(\tau,X)$. The $((\tau,X))\vert _{\tau_i = v} w$ notion presents the number of occurrences of pattern $(\tau,X)$ in $w$ such that in each occurrence, the entry presents for the $\tau_i$ is $v$.

We define this annotation more formal: Given a word $w$ and the vincular pattern $(\tau,X)$. Let $I = \{j:\tau_j = \tau_i\}$. The $((\tau,X))\vert _{\tau_i = v} w$ notion presents the number of occurrences of pattern $(\tau,X)$ in $w$ such that the occurrence $\omega$ has $\omega(j) = v \forall j \in I$.
\end{Def}
\begin{Exam}
Given a word $w=w_1\dots w_n$. Annotation $\underline{31}2\vert _{3 = 5}$ presents the tuple $(i,j)$ such that $w_i > w_j > w_{i+1}$ and $w_i = 5$. $(\underline{31}2)\vert _{3 = 5}w$ is the number of these occurrences in $w$. For example, if $w = 65253341$, the multiset of occurrences of $\underline{31}2\vert _{3 = 5}$ is $\{\underline{52}3,\underline{52}3,\underline{52}4,\underline{53}4\}$. Hence $(\underline{31}2)\vert _{3 = 5}w = 4$.

Consider another example with $\underline{21}2\vert _{2 = 4}$. This notion presents the tuple $(i,j)$ such that $w_i > w_{i+1}$ and $w_i = w_j = 4$. If $w = 43214414$, the multiset of occurrences of $\underline{21}2\vert _{2 = 4}$ is $\{\underline{43}4,\underline{43}4,\underline{43}4,\underline{41}4\}$. Hence $(\underline{21}2)\vert _{2 = 4}w = 4$.
\end{Exam}

\vspace{5mm}
\noindent
From \cref{def4} to \cref{def7}, we define annotations relate to \textbf{4-tuple-letter}
\begin{Def}\label{def4}
A \textit{4-tuple-letter} is defined as
\begin{align*}
    (\v,\d,\p,\r): \mathbb{N} \times \mathbb{N}^{\*} \times \{\text{opener},\text{closer},\text{insider},\text{outsider}\} \times \mathbb{N}
\end{align*}

Given a word $w=w_1w_2\dots w_n$ in $\Sset_{(A,m)}$. For all letter $w_i \in w$, its 4-tuple-letter $W_i$ is $(\v(w_i),\d(w_i),\p(w_i),\r(w_i))$.

The value, duplicate index, position, and right embracing number of a 4-tuple-letter $W_i$ is the first, second, third, and fourth entry of tuple $W_i$, respectively. The annotation $\v,\d,\p,\r$ can also be used for denoting corresponding concepts of the 4-tuple-letter.
\end{Def}
\begin{Def}
The \textit{set of 4-tuple-letter} of $w$ is the set of all 4-tuple-letter of letter in $w$
\end{Def}

\begin{Def}
Given a tuple of 4-tuple-letter $T = (W_1,\dots ,W_n)$. The \textit{tuple of non-descent-top status} of $T$ is $(x_1,\dots ,x_n)$, in which $x_i = 1$ if and only if $W_i$ is the non descent top, otherwise $x_i = 0$. The \textit{tuple of right embracing number} of $T$ is $(\r(W_1),\dots ,\r(W_n))$, in which $\r(W_i)$ is the right embracing number of $W_i$
\end{Def}

\begin{Rem}\label{rem2}
Based on the way we define $W_i$, it is clear that with the specific $i_0$, the tuple $(\v(W_{i_0}),\d(W_{i_0}))$ is unique in the set $\{(\v(W_i),\d(W_i)): W_i \in W\}$. For that reason, we can use annotation $W_{(\v(W_i),\d(W_i))}$ for presenting the $W_i$ block in $W$ or for the letter $w_i$ in word $w$
\end{Rem}
\begin{Rem}\label{rem3}
With $W_{(v,i)}$ and $W_{(v,j)}$ in 4-tuple-letter set $W$ of the word $w$, if $i<j$, then $\r(W_{(v,i)}) \geq \r(W_{(v,j)})$
\end{Rem}
\begin{Exam}\label{ex1}
Given a word $w=421-4-43$. Below is the table for value, duplicate index, position, and right embracing number of all letters of $w$.
\[\begin{matrix}
    w(i) & 4_1 & 3_1 & 1_1 & 4_2 & 4_3 & 2_1 \\
    & W_1 & W_2 & W_3 & W_4 & W_5 & W_6 \\
    \v   & 4 & 3 & 1 & 4 & 4 & 2\\
    \d & 1 & 1 & 1 & 2 & 3 & 1\\
   \p & \text{closer} & \text{insider} & \text{opener} & \text{outsider} & \text{closer} & \text{opener}\\
   \r & 0 & 1 & 0 & 0 & 0 & 0
\end{matrix}
\]
The 4-tuple-letter set of $w$ is \begin{multline*}
    W = \{(4,1,\close,0), (3,1,\insi,1), (1,1,\open,0), (4,2,\out,0), \\(4,3,\close,0), (2,1,\open,0)\}
\end{multline*}
\end{Exam}
\begin{Def}\label{def5}
We denote $L_{\Sset_{(A,m)}} = \{\text{4-tuple-letter set of } w: w \in \Sset_{(A,m)}\}$ is the set of all 4-tuple-letter set of multiset permutation in $\Sset_{(A,m)}$. A 4-tuple-letter set $W$ is a \textit{consistent 4-tuple-letter set} if exists a multiset $(W^{\prime},m^{\prime})$ such that $W \in L_{\Sset_{(W^{\prime},m^{\prime})}}$
\end{Def}
\begin{Def}\label{def6}
Given a 4-tuple-letter set $W=\{W_1, W_2, \dots, W_k\}$. We say $W$ is a \textit{$v$-consistent 4-tuple-letter set} when:
\begin{itemize}
    \item $\forall W_i \in W$, $\v(W_i) = v$
    \item $\forall j\in [k]$, $\exists W_i\in W:\d(W_i)=j$
    \item if $\d(W_i)<\d(W_j)$ then $\r(W_i) \leq \r(W_j)$
\end{itemize}
\end{Def}
\begin{Def}\label{def7}
Given a 4-tuple-letter set $W=\{W_1, W_2, \dots, W_k\}$. A \textit{multiset of descent top} of $W$ is $\Destop(W) = \{\v(W_i) : \p(W_i) = \text{insider} \text{ or} \text{ closer}\}$. It is obvious that if $W$ is the set of 4-tuple-letter of $w$, $\Destop(W)$ is also the multiset of descent top of $w$. Respectively, we define a \textit{multiset of descent bottom} of $W$ is $\Desbot(W) = \{\v(W_i) : \p(W_i) = \text{insider} \text{ or} \text{ opener}\}$.
\begin{Exam}
Consider $W$ as result from the example 1, then $\Destop(W) = \{5,4,4,5\}$
\end{Exam}
\end{Def}
\vspace{5mm}
In the \cite{CLARKE1997237}, the author introduced the term $v$-skeleton and how to apply it in recovering the permutation from letters with known position and right embracing number. We extend their solutions into $v_i$-insertion to recover the word $w$ from the set of 4-tuple-letter.

From \cref{def8} to \cref{def13}, we define annotations relate to \textbf{$n_i$-insertion}. Some of these terms were introduced in the \cite{CLARKE1997237}

\begin{Def}\label{def8}
Given a word $w \in \Sset_{(A,m)}$. A \textit{block} is a subset $B$ of $(A,m) \cup \infty$ such that $B \cap (A,m) \neq \emptyset$. The block $B$ is called \textit{open} if $\infty \in B$, \textit{closed} if $\infty \notin B$, and \textit{improper} if $\mid B\mid =1$. At the time constructing $v_i$-insertion, the \textit{$v$ embraceable open block}, abbreviated by $v\eob$, is the open block that has the letter next to $\infty$ is not $v$. For example, if we are constructing $4_1$-insertion, then $\infty 4 1$ is not the $4$ embraceable open block, but $\infty 3 1$ is.
\end{Def}
\begin{Def}
An \textit{insertion} is a sequence $S=B_1-B_2\dots B_r$ of blocks. The insertion $S$ is valid if for each $i$ with $1 \leq i < n$ we have $O(B_i) < C(B_i)$.
\end{Def}
\begin{Def}\label{def9}
Let $w$ be a permutation with descent block decomposition is $B_1-\dots-B_r$. Let $v$ is the letter in the multiset $M(w)$. The $v_i$-insertion of $w$ is the sequence of blocks obtained by 
\begin{itemize}
    \item deleting any descent block $B$ of $w$ for which
    \begin{itemize}
        \item $O(B)>v$, or
        \item $O(B)=v$ and $B$ is to the right of the $i^th$ letter $v$ of $w$
    \end{itemize}
    \item replacing any remaining letter of $w$ that is
    \begin{itemize}
        \item greater than $v$, or
        \item equal to $v$ and that letter is to the left of the $i^th$ letter $v$ of $w$
    \end{itemize}
    by $\infty$
    \item replacing each remaining descent block by its underlying set
\end{itemize}
\end{Def}
\begin{Exam}
With the word $w=5_12_1-3_1-5_23_2-4_11_1-5_33_31_2$, the $3_2$-insertion of $w$ is the sequence $\{2_1,\infty\}-\{3_2,\infty\}-\{1_1,\infty\}-\{1_2,3_3,\infty\}$, which will be written as $\infty2_1-\infty3_2-\infty1_1-\infty3_31_2$
\end{Exam}
\begin{Def}\label{def10}
The insertion $S$ is called \textit{valid} when:
\begin{itemize}
    \item $\forall$ block $B \in S$, $O(B) < C(B)$
    \item $\forall v_i$ and $v_j\in S$, if $i < j$, then $v_i$ is to the left of $v_j$ in $S$-insertion
\end{itemize}
\end{Def}
\begin{Rem}
All the $v_i$-insertion of the word $w$ is valid
\end{Rem}
\begin{Def}\label{def11}
The \textit{well-defined 4-tuple-letter set} $W=\{W_1,W_2,\dots,W_n\}$ is a set of 4-tuple-letter such that:
\begin{itemize}
    \item can find the position to insert new block in $\zeta$ induction step
    \item $\mid \Destop(W)\mid = \mid\Desbot(W)\mid$
    \item $\forall v \in A, W\vert _{\v=v}=\{W_i:\v(W_i)=v\}$  is $v$-consistent-4-tuple-letter set, in which $A$ is the set formed from distinct $\v(W_i)$ in $W$
\end{itemize}
\end{Def}
\begin{Def}\label{def12}
Given $S$ which is a valid $v_{i-1}$-insertion or the ${v-1}_j$-insertion. The 4-tuple-letter $W_{(v,i)}$ is said \textit{constructible on $S$} if we can construct valid $v_i$-insertion based on $S$ and $W_{(v,i)}$
\end{Def}
\begin{Def}\label{def13}
Given a valid insertion $S$ such that $S$'s largest letter is $v-1$. Given a $v$-consistent 4-tuple-letter set $W\vert _{\v=v}={W_{(v,1)}, \dots, W_{(v,k)}}$. $W$ is said \textit{constructible on $S$} if:
\begin{itemize}
    \item $W_{(v,k)}$ is constructible on $S$
    \item $\forall i \in \overline{1,k-1}$, $W_{(v,i)}$ is constructible on $v_{i+1}$-insertion
\end{itemize}
\end{Def}

\section{Bijection on repetitive permutation set transforms $\mad$ and $\madl$ statistics}
In this section, we only consider the correctly indexed word as restricted in the \cref{index}. The resulted word of the bijection will also be proved to be correctly indexed.

\paragraph{$\Phi$ function} Given a word $w$ over $\Sset_{(A,m)}$. We construct the function $\Phi$ maps the word $w$ to the word $w^{\prime} = \Phi(w)$ such that:
\begin{equation*}
    ((2\underline{31}),(\underline{31}2),(\underline{21}1),(1\underline{21}),(\underline{21}2),(2\underline{21}),\underline{21})w = ((2\underline{31}),(\underline{31}2),(\underline{21}1),(1\underline{21}),(2\underline{21}),(\underline{21}2),\underline{21})\Phi(w)
\end{equation*}
In order to achieve $(\underline{21}2)w = (2\underline{21})\Phi(w)$, the $\Phi$ will be constructed such that for all $v$ in $M(w)$:
\begin{equation*}
    (\underline{21}2\vert _{2=v})w = (2\underline{21}\vert _{2=v})\Phi(w)
\end{equation*}
and
\begin{align*}
    (\underline{21}1\vert _{1=v})w &= (\underline{21}1\vert _{1=v})\Phi(w)\\
    (1\underline{21}\vert _{1=v})w &= (1\underline{21}\vert _{1=v})\Phi(w)
\end{align*}

\noindent
The function $\Phi$ consists of 3 main functions:
\begin{itemize}
    \item function $\delta$ splits the word into the set of 4-tuple-letters
    \item function $\epsilon$ maps the set of 4-tuple-letter of $w$ with the set of 4-tuple-letter of $w^{\prime}$
    \item function $\zeta$ recover the word from the set of 4-tuple-letters
\end{itemize}
\begin{equation*}
    w^{\prime} = \Phi(w) = \zeta \circ \epsilon \circ \delta (w)
\end{equation*}
\begin{figure}[!h]
\begin{center}
    \begin{tikzpicture}[line cap=round,line join=round,>=triangle 45,x=1cm,y=1cm]
\draw (-0.5,5) node[anchor=north west] {$\{W_1,W_2,\dots ,W_n\}$};
\draw (0.9,9) node[anchor=north west] {$w$};
\draw (5.7,9.1) node[anchor=north west] {$w^{\prime}$};
\draw (3.4,9.5) node[anchor=north west] {$\Phi$};
\draw (4.7,5) node[anchor=north west] {$\{U_1,U_2,\dots ,U_n\}$};
\draw (6.1,7) node[anchor=north west] {$\zeta = \delta^{-1}$};
\draw (0.5,7) node[anchor=north west] {$\delta$};
\draw (3.5,5.1) node[anchor=north west] {$\epsilon$};
\draw [->,line width=1pt] (1.2,8.2) -- (1.2,5);
\draw [->,line width=1pt] (2.8,4.7) -- (4.7,4.7);
\draw [->,line width=1pt] (6,5.2) -- (6,8.2);
\draw [dashed,->,line width=1pt] (1.7,8.8) -- (5.6,8.8);
    \end{tikzpicture}
\end{center}
\label{PhiBijection}
\caption{$\Phi$ flow}
\end{figure}
These functions will be introduced in the following sub-sections. We will also prove that they are bijective.

\subsection{Function $\delta$}
\paragraph{$\delta$ function}Given a word $w=w_1w_2\dots w_n$. The function $\delta$ maps the word $w$ to the consistent 4-tuple-letter set $W = \{W_1,W_2,\dots ,W_n\}$, in which $W_i$ is the 4-tuple-letter of $w_i$.
\begin{alignat*}{4}
& \delta:   \quad && \Sset_{(A,m)} && \to && \quad \quad L_{\Sset_{(A,m)}}\\
&    \quad && \quad w \quad && \mapsto \quad && \{W_1,W_2,\dots,W_n\}
\end{alignat*}
The following lemma is obvious
\begin{Lem}\label{lma21}
Given $w \in \Sset_{(A,m)}$. Let $W = \{W_1,\dots,W_n\}$ be an image or the word $w$ under $\delta$, $W\vert _{\v=v}$ be a set of all letters that has value is $v$ in $W$. Then $\forall v \in A$, $W\vert _{\v=v}$ is the $v$-consistent 4-tuple-letter set.
\end{Lem}

\pagebreak
\subsection{Function $\zeta$}
Before introducing the $\zeta$ function, we have following remarks:
\begin{Lem}\label{lma14}
The right embracing number of $v_i$ in $w$ equals the number of $v$ embraceable open block to the right of the block containing $v_i$ in the $v_i$-insertion of $w$.
\end{Lem}
\begin{Lem}\label{lma15}
The duplicate index of $v_i$ in $w$ equals the difference between the total number of letter $v$ in $w$ and the number of letters $v$ that are to the right of the block containing $v_i$ in the $v_i$-insertion of $w$.
\end{Lem}

\paragraph{$\zeta$ function}\label{zeta}
We shall describe the function $\zeta$ which maps the consistent set of 4-tuple-letter $W=\{W_1,W_2,\dots,W_n\}$ to the word $w^{\prime}$ such that $\delta(w^{\prime}) = W$.
\begin{alignat*}{4}
& \zeta:   \quad && \quad \quad L_{\Sset_{(A,m)}} && \to && \quad \Sset_{(A,m)}\\
&    \quad && \{W_1,W_2,\dots,W_n\} \quad && \mapsto \quad && \quad w
\end{alignat*}

\noindent
The insertion will be constructed inductively, follow the increasing order of value and the decreasing order of the duplicate index. After the $v_i$-insertion is constructed, if $i \neq 1$, then the $v_{i-1}$-insertion will be constructed based on $v_i$-insertion, otherwise ${(v+1)}_{k_{v+1}}$-insertion, in which $k_{v}$ is the number of 4-tuple-letter having value is $v$ in $W$.

We will show that the right embracing number and duplicate index are ensured in each induction case. Later on, in the \cref{lma13}, we will also prove that we can execute the consistent 4-tuple-letter set under $\zeta$, and the position defined by the algorithm is the only position that we can insert a new letter into.

The $1_{k_1}$-insertion of $w$ is either $\{1\}$ or $\{1,\infty\}$, according to whether $1_{k_1}$ is an outsider or an opener. Suppose that the $v_i$-insertion $S=B_1-B_2\dots B_r$ of $w$ has been constructed.
\begin{itemize}
    \item Consider the case the next insertion to be constructed is $v_{i-1}$-insertion. To construct the $v_{i-1}$-insertion of $w$, we must insert $v_{i-1}$ in the correct place in $S$. Let the embracing number of $W_{(v,i-1)}$ in $w$ be $e$. Let $B_i$ and $B_j$ denotes the $e^{\text{th}}$ and the $(e+1)^{\text{th}}$ $v$-embraceable open block from the right in $S$.
    \begin{itemize}
        \item If $W_{(v,i-1)}$ is an insider or a closer in $w$ then $v_{i-1}$ must be inserted into $B_j$, and if $W_{(v,i-1)}$ is a closer then $B_j$ must be closed by the removal of its $\infty$. \numbAtBack{$\zeta-1$}
        \item We consider if $W_{(v,i-1)}$ is an outsider. Let $B^{\prime}$ be a block that contains $v_i$. If the block $B_i$ is to the left of $B^{\prime}$, we will insert improper block $\{v_{i-1}\}$ immediately to the left of $B_i$.
        Otherwise, we insert improper block $\{v_{i-1}\}$ immediately to the left of $B^{\prime}$. \numbAtBack{$\zeta-2$}
        \item Similar to the case of outsider, if $W_{(v,i-1)}$ is an opener, we have to determine whether the new block should be inserted immediately after $B_i$ or $B^{\prime}$. The block to be inserted is the open block $\{v_{i-1},\infty\}$. \numbAtBack{$\zeta-3$}\\
    \end{itemize}
    \item Consider the case the next insertion to be constructed is $(v+1)_{k_{v+1}}$-insertion. Let the embracing number of $W_{(v+1)_{k_{v+1}}}$ in $w$ be $e$. Let $B_i$ and $B_j$ denotes the $e^{\text{th}}$ and the $(e+1)^{\text{th}}$ $v$-embraceable open block from the right in $S$.
    \begin{itemize}
        \item If $W_{(v+1)_{k_{v+1}}}$ is an insider or a closer in $w$ then $(v+1)_{k_{v+1}}$ must be inserted into $B_j$, and if $W_{(v+1)_{k_{v+1}}}$ is a closer then $B_j$ must be closed by the removal of its $\infty$.\numbAtBack{$\zeta-4$}
        \item If $W_{(v+1)_{k_{v+1}}}$ is an outsider then the improper block $\{(v+1)_{k_{v+1}}\}$ must be inserted immediately to the left of $B_i$\numbAtBack{$\zeta-5$}
        \item Similar to the case of outsider, if $W_{(v+1)_{k_{v+1}}}$ is an opener, we have to insert the open block $\{(v+1)_{k_{v+1}},\infty\}$ immediately to the left $B_i$.\numbAtBack{$\zeta-6$}
    \end{itemize}
\end{itemize}

\begin{Exam}\label{ex3}
From a 4-tuple-letter set $W = \{(4,1,\close,0), (3,1,\insi,1), (1,1,\open,0),$ $ (4,2,\out,0), (4,3,\close,0), (2,1,\open,0)\}$, we apply $\zeta$ to construct word $w$:
\begin{itemize}
    \item $1_1$-insertion: $\infty 1_1$
    \item $2_1$-insertion: $\infty 1_1 - \infty 2_1$
    \item $3_1$-insertion: $\infty 3_1 1_1 - \infty 2_1$
    \item $4_3$-insertion: $\infty 3_1 1_1 - 4_3 2_1$
    \item $4_2$-insertion: $\infty 3_1 1_1 - 4_2 - 4_3 2_1$
    \item $4_1$-insertion: $4_1 3_1 1_1 - 4_2 - 4_3 2_1$
\end{itemize}
\end{Exam}

\begin{The}\label{lma13}
Given a word $w$ with its set of 4-tuple-letter $W$. Then:
\begin{itemize}
    \item $W$ is executable under $\zeta$
    \item $\zeta(W) = w$
\end{itemize}
\end{The}
\begin{proof}
We will use the same annotation as the $\zeta$ description in this proof.

It is clear that in the word $w$, letters $1$ will always be outsider or opener and have the right embracing numbers are $0$. So that the $1_{k_1}$-insertion of $\zeta(W)$ can be constructed, and equals to $1_{k_1}$-insertion of $w$

Assume that we can construct $v_i$-insertion from $W$ under $\zeta$ and the $v_i$-insertion of $\zeta(W)$ equals to $v_i$-insertion of $w$.
\begin{itemize}
    \item We consider the case when we construct $v_{i-1}$-insertion based on $v_i$-insertion. We shall prove that we can find the position in $v_i$-insertion to insert $v_{i-1}$ (a), and the $v_{i-1}$-insertion of $\zeta(W)$ equals to $v_{i-1}$-insertion of $w$ (b). In order to prove (b), we will find the only possible position for block containing $v_{i-1}$ in $v_{i-1}$-insertion of $w$. We also need to prove the found position satisfy the constraint about embracing value and duplicate index.
    \begin{itemize}
        \item Consider case ($\zeta_1$)\\
        Let $B$ denote the block contains $v_{i-1}$ in $v_{i-1}$-insertion. If $v_{i-1}$ is the insider or closer, then the $v_{i-1}$ insertion of $w$ can be created by removing $v_{i-1}$ from $B$, and add $\infty$ into $B$ if $v_{i-1}$ is the closer.
    
        Follow the \cref{lma14}, there is exactly $\r(v_{i-1})$ $v$-embraceable open blocks to the right of $B$ in $v_i$-insertion. So in the $v_i$-insertion, which equals to $v_i$-insertion of $\zeta(W)$, there exists a block $B^{\dprime}$, which is the block $B$ before inserting $v_{i-1}$, such that it is the $\l(v_{i-1}) + 1$-th $v$-embraceable open blocks count from the right. $B^{\dprime}$ is the $B_i$ block we find in the $\zeta$ description.
    
        We remark that since $W\vert _{\v=v}$ is $v$-consistent set of 4-tuple-letter set, $W_{(v,i-1)}$ has its right embracing numbers larger or equal to the right embracing numbers of other 4-tuple-letters of $W\vert _{\v=v}$ inserted into $v_i$-insertion. So that $B_j$ block is to the left of all inserted $v$ letters. Hence the $v$ letters in $v_{i-1}$-insertion are ordered to follow the increasing order of their duplicate indices.\\
        It is clear that:
        \begin{align*}
            \text{$v_{i-1}$-insertion of $w$} &= \text{insert $v_{i-1}$ into the $v_i$-insertion}\\
            &= \text{insert $v_{i-1}$ into the $v_i$-insertion of $\zeta(W)$}\\
            &= \text{$v_{i-1}$-insertion of $\zeta(W)$}
        \end{align*}
        (a) and (b) are correct in this case.
        \item Consider case ($\zeta_1$)\\
        Let $B = \{v_{i-1}\}$ denote the block contains $v_{i-1}$ in $v_{i-1}$-insertion. When $v_{i-1}$ is the outsider, the $v_{i-1}$ insertion of $w$ can be created by removing $B$ from $v_i$-insertion.
    
        Follow the \cref{lma14} and \cref{lma15}, we can find the block $B_i$ and the block $B^{\prime}$ in $v_i$-insertion. These 2 blocks remain the same in $v_{i-1}$-insertion. We consider 2 cases:
        \begin{itemize}
            \item when $B_i$ is to the left of $B^{\prime}$ in $v_{i-1}$-insertion
            
            In $v_{i-1}$-insertion of $w$, $B$ must be immediately to the left of $B_i$. Other cases are:
            \begin{itemize}
                \item when $B$ is to the right of $B_i$\\
                In this case, the number $v$-embraceable open block to the right of $B$ will smaller than $\r(W_{(v,i-1)})$, which according to the \cref{lma14}, implies that $\r(v_{i-1})$ in $v_{i-1}$-insertion will be smaller than $\r(W_{(v,i-1)})$. Since this is a contradiction the setting $\r(W_{(v,i-1)}) = \r(v_{i-1})$, this case is eliminated.
                \item when $B$ is to the left of $B_i$ but not immediately before $B_i$\\
                This assumption is equivalent to existing block $B^{\dprime}$ in the left of $B_i$ such that $B^{\dprime}$ is immediately to the right of $v_{i-1}$. We notice that since when $B_i$ is to the left of $B^{\prime}$, all the $v$ letters in $v_{i-1}$-insertion besides $v_{i-1}$ are to the right of $B_i$.\\
                If $B^{\dprime}$ is the open-block, then $B^{\dprime}$ must be a $v$-embraceable open block. Follow the \cref{lma14}, the right embracing number of $v_(i-1)$ will be greater than $\r(W_{(v,i-1)})$, which is a contradiction the setting $\r(W_{(v,i-1)}) = \r(v_{i-1})$\\
                If $B^{\dprime}$ is the closed-block, then $B^{\dprime}-C(B^{\dprime})$ will create another descent pair. This is contradict the setting $W_{(v,i-1)}$ is outsider.
            \end{itemize}
            \item when $B_i$ is to the right of $B^{\prime}$ in $v_{i-1}$-insertion\\
            Then in $v_{i-1}$-insertion of $w$, $B$ must be immediately to the left of $B^{\prime}$. Otherwise:
            \begin{itemize}
                \item when $B$ is to the right of $B^{\prime}$\\
                Follow the \cref{lma15}, the duplicate index of $v_(i-1)$ will be smaller than $i-1$. We reject this case.
                \item when $B$ is to the left of $B^{\prime}$ but not immediately before $B_i$\\
                In this case, all the $v$ letters in $v_{i-1}$-insertion besides $v_{i-1}$ are to the right of $B_i$. This case can be explained why we reject it similarly to the second case of when $B_i$ being to the left of $B^{\prime}$ in $v_{i-1}$-insertion.
            \end{itemize} 
        \end{itemize}
    
        We will show that $v$ letters in $v_{i-1}$-insertion have correct $\r$, and are ordered to follow the increasing order of the duplicate index.
        \begin{itemize}
            \item In the case when the block $B_i$ is to the left of $B^{\prime}$, since $B$ is inserted immediately to the left of $B_i$, the number $v$-embraceable open block to the right of $B$, which equals to $\r(v_{i-1})$ in $v_{i-1}$-insertion, will equals to $\r(W_{(v,i-1)})$.
            \item We consider the case the block $B_i$ is to the right of $B^{\prime}$. Since $B_i$ is to the right of $B^{\prime}$, follows \cref{lma14}, $\r(v_i) \geq \r(v_{i-1})$. Since $W\vert _{\v=v}$ is $v$-consistent set of 4-tuple-letter set, $\r(W_{(v,i)}) \leq \r(W_{(v,i-1)})$. Hence $\r(v_i) = \r(v_{i-1})$, which means there is exactly $\r(v_{i-1})$ $v$-embraceable open block to the right of $B^{\prime}$. Inserting $B$ immediately to the left of $B^{\prime}$ makes $\r(v_{i-1})$ in $v_{i-1}$-insertion equals to $\r(W_{(v,i-1)})$.
        \end{itemize}
        In both cases, since $B$ is inserted to the left of $B^{\prime}$, which is the blocks contain $v_i$, the block sequence from $B^{\prime}$ to the end of $v_{i-1}$-insertion remains. Hence the order between pre-added $v$ letters and their $\r$ values remain. Adding $v_{i-1}$, the one that has the lowest duplicate index for now, to the left of $B^{\prime}$ makes the $v$ letter still order in the increasing order of the duplicate index.
    
        The $\zeta$ correctly chooses the only possible position for inserting $B$ into $v_i$ insertion. (a) and (b) are correct in this case.
        \item Consider case ($\zeta_3$):
        Since $v_{i-1}$ is the opener, so the block contains $v_{i-1}$ in $v_{i-1}$-insertion $B$ will equals to $\{v_{i-1},\infty\}$
    
        Similarly to the proof when considering case ($\zeta_2$), we can prove that the position for inserting $B$ in the $\zeta$ function is the only possible position for inserting $B$ into $v_i$-insertion. It implies that (a) and (b) are correct in this case.
    \end{itemize}

    \item Next, we consider the case when we construct $v+1_{k_{v+1}}$-insertion based on $v_i$-insertion. We shall prove that we can find the position in $v_i$-insertion to insert $v+1_{k_{v+1}}$ (c), and the $v+1_{k_{v+1}}$-insertion of $\zeta(W)$ equals to $v+1_{k_{v+1}}$-insertion of $w$ (d).
    
    The proof for (c) and (d) of ($\zeta_4$), ($\zeta_5$), ($\zeta_6$) sub-cases is similar to the proof of (a) and (b) of ($\zeta_1$), ($\zeta_2$), ($\zeta_3$) sub-cases, respectively. And since $(v+1)_{k_{v+1}}$ is the first $v+1$ letter inserted in $\zeta$, the condition about order between $v+1$ letter for the valid insertion is ensured.
\end{itemize}
From the inductive proof mentioned above, it is clear that:
\begin{align*}
    \begin{cases}
    \forall W_{(v,i)} \in \delta(w), \text{ we can construct $v_i$-insertion from $W$ under $\zeta$} & \iff \delta(w) \text{ is valid}\\
    \forall W_{(v,i)} \in \delta(w), \text{ $v_i$-insertion of $\zeta(W)$ $=$ $v_i$-insertion of $w$} & \implies \zeta(\delta(w)) = w
    \end{cases}
\end{align*}
\end{proof}

\begin{The}\label{the1}
The function $\delta: \Sset_{(A,m)} \to L_{\Sset_{(A,m)}}$ is bijective.
\end{The}
\begin{proof}
According to the way we define $L_{\Sset_{(A,m)}}$ and $\delta$, it is clear that $\delta$ is surjective.\numbAtBack{1}

\noindent
Assume that there exists $w_1 \in \Sset_{(A,m)}$ and $w_2 \in \Sset_{(A,m)}$ such that:
\begin{align}
    \label{ethe11}
    w_1 &\neq w_2\\
    \label{ethe12}
    \delta(w_1) &= \delta(w_2) = W
\end{align}
Follow the \cref{lma13}, $\forall u \in \Sset_{(A,m)}$, $\zeta(\delta(u)) = u$. Apply this to $w_1$ and $w_2$:
\begin{align} \label{ethe13}
\nonumber
    \zeta(\delta(w_1)) = w_1\\
    \zeta(\delta(w_2)) = w_2
\end{align}
Since \cref{ethe11} and \cref{ethe13}
\begin{equation}
    \label{ethe14}
    \zeta(\delta(w_1)) \neq \zeta(\delta(w_2))
\end{equation}
Apply \cref{ethe12} into \cref{ethe14}, we have $\zeta(W) \neq \zeta(W)$. Hence, the assumption is incorrect. So that the $\delta$ is injective.\numbAtBack{2}

\noindent
From (1) and (2), $\delta$ is proved to be bijective
\end{proof}

\begin{The}\label{the2}
The function $\zeta: L_{\Sset_{(A,m)}} \to \Sset_{(A,m)}$ is bijective.
\end{The}
\begin{proof}
Since $\delta$ is bijective:
\begin{align*}
    \mid \Sset_{(A,m)}\mid  = \mid L_{\Sset_{(A,m)}}\mid 
\end{align*}
Since domain and codomain of $\zeta$ have the same cardinality, if we can prove that $\zeta$ is surjective, then the lemma is proved.

\noindent
Assume that there exists a word $w_0$ in $\Sset_{(A,m)}$ such that there is no preimage of $w_0$ in $L_{\Sset_{(A,m)}}$. The \cref{lma13} states that $\forall u \in \Sset_{(A,m)}$, $\zeta(\delta(u)) = u$, hence $\zeta(\delta(w_0)) = w_0$. So $w_0$ still has preimage under $\zeta$ function, which is $\delta(w_0)$. The assumption is incorrect, which means $\zeta$ is surjective. Hence it is bijective.
\end{proof}

After proving $\delta$ and $\zeta$ is bijective and having relation $\forall u \in \Sset_{(A,m)}$, $\zeta(\delta(u)) = u$, it is clear that:
\begin{Lem}\label{lma24}
The function $\zeta$ is an inverse of the function $\delta$
\end{Lem}

\vspace{5mm}
\paragraph{Properties of the consistent 4-tuple-letter set} In this part, we shall point out some properties of the consistent 4-tuple-letter set, which will be used in the next sub-section.
\begin{Lem}\label{lma22}
Given a well-defined 4-tuple-letter set $W$. Then $W$ is also the consistent 4-tuple-letter set.
\end{Lem}
\begin{proof}
We review all the conditions of the definition of a well-defined 4-tuple-letter set. The first condition ensures that $W$ can be executed by $\zeta$. We define a new function $\zeta^{\prime}$ that has the same rule as $\zeta$ but its domain is a set of all possible well-defined 4-tuple-letter set. Let $S$ be an image of $W$ under $\zeta^{\prime}$. By following the restriction about choosing the position of $\zeta$, the right embracing number and the position of $W_{(v,i)}$ equal to these characteristics of $v_i$ in $S$.

\noindent
The second condition ensures that $S$ contains no open blocks. It implies that $S$ can be presented as a word.

\noindent
The last condition ensures the restriction about duplicate index for the valid insertion can be met; and the resulted word has letters are correctly indexed (starts with 1, increase duplicate index one by one).

These observations show that there exists the correctly indexed word $w$, which is $\zeta^{\prime}(W)$, in which the position, the duplicate index, and the right embracing number of its letter $v_i$ equals to these characteristics of $W_{(v,i)}$. Hence $W$ is a consistent 4-tuple-letter set.
\end{proof}
\begin{Lem}\label{lma23}
Given a consistent 4-tuple-letter set $W$. Then $W$ is well-defined
\end{Lem}
\begin{proof}
Follows the definition of well-defined 4-tuple-letter set, $U$ is a well-defined 4-tuple-letter set if and only if:
\begin{itemize}
    \item can find the position to insert new block in $\zeta$ induction step
    
    $W$ satisfies this condition according to \cref{lma13}.
    \item $\mid\Destop(W)\mid = \mid\Desbot(W)\mid$
    
    $W$ satisfies this condition according to \cref{lma3}.
    \item $\forall v \in $A$, W\vert _{\v=v}=\{W_i:\v(W_i)=v\}$ is $v$-consistent-4-tuple-letter set, in which $A$ is formed from distinct $\v(W_i)$ in $W$
    
    $W$ satisfies this condition according to \cref{lma21}.
\end{itemize}
\end{proof}
From \cref{lma22} and \cref{lma23}, it is clear that:
\begin{Rem}\label{rem8}
The 4-tuple-letter set $W$ is well-defined $\iff$ the 4-tuple-letter set $W$ is a consistent 4-tuple-letter set.
\end{Rem}
\vspace{3mm}
Given a 4-tuple-letter set $W$. We consider the conditions of a well-defined 4-tuple-letter set:
\begin{enumerate}
    \item can find the position to insert new block in $\zeta$ induction step
    \item $\mid \Destop(W)\mid = \mid \Desbot(W)\mid$
    \item $\forall v \in A, W\vert _{\v=v}=\{W_i:\v(W_i)=v\}$  is $v$-consistent-4-tuple-letter set, in which $A$ is the set formed from distinct $\v(W_i)$ in $W$
\end{enumerate}
From the proof of \cref{lma13}, it is clear that with any 4-tuple-letter set $U$ that satisfies conditions (1) and (2), $\zeta(U)$ will be the valid insertion

\noindent
Now we have a look at the conditions for 4-tuple-letter $U$ is constructible on $S$ insertion
\begin{enumerate}
    \item $S$ is a valid insertion
    \item $U$ is a consistent 4-tuple-letter
    \item all letters in $S$ have smaller value than a value of any 4-tuple-letters in $U$
    \item can find the position to insert new block in $\zeta$ induction step
    \item the result insertion after inserting $U$ into $S$ is valid
\end{enumerate}

\begin{Rem}\label{rem5}
From the remarks above, it is clear that we can prove a 4-tuple-letter set $W$ is well-defined by proving
\begin{enumerate}
    \item $\mid \Destop(W)\mid = \mid \Desbot(W)\mid$
    \item $\forall v \in A, W\vert _{\v=v}=\{W_i:\v(W_i)=v\}$  is $v$-consistent-4-tuple-letter set
    \item $\forall v \in A, U\vert _{\v=v} \text{ is constructible on }{v-1}_{k_{v-1}}-\text{insertion}$
\end{enumerate}
, in which $A$ is the set formed from distinct $\v(W_i)$ in $W$
\end{Rem}
We will apply the \cref{rem5} in \cref{lma20} to prove that with given well-defined 4-tuple-letter $W$, $\epsilon(W)$ is also well-defined. The \cref{lma16} supports for this proof.
\begin{Lem}\label{lma17}
If the $v_{i-1}$-insertion is constructible on $v_i$-insertion, then:
\begin{align*}
        \#v\eob(\text{$v_{i-1}$-insertion}) = \#v\eob(\text{$v_i$-insertion}) - \mid \Destop(\{W_{(v,i-1)}\}\mid
    \end{align*}
\end{Lem}
\begin{proof}
We consider following cases:
\begin{itemize}
    \item when $W_{(v,i-1)}$ is insider or closer $\iff$ $W_{(v,i-1)}$ is descent top\\
    In order to create $v_{i-1}$-insertion, $v_{i-1}$ will be inserted into a $v$-embraceable open block in $v_i$-insertion, hence the number of $v$-embraceable open block will be reduced by 1
    \begin{align*}
        \#v\eob(\text{$v_{i-1}$-insertion}) &= \#v\eob(\text{$v_i$-insertion}) - 1\\
        \mid \Destop(\{W_{(v,i-1)}\}\mid &= 1
    \end{align*}
    \item when $W_{(v,i-1)}$ is closer or outsider $\iff$ $W_{(v,i-1)}$ is non descent top\\
    In this case, the new block contains $v_{i-1}$, which cannot embrace $v$, will be inserted. Hence the number of $v$-embraceable open block remains.
    \begin{align*}
        \#v\eob(\text{$v_{i-1}$-insertion}) &= \#v\eob(\text{$v_i$-insertion})\\
        \mid \Destop(\{W_{(v,i-1)}\}\mid &= 0
    \end{align*}
\end{itemize}
It is clear that the equation in the lemma is correct in all cases.
\end{proof}

\begin{Lem}\label{lma16}
Given a valid insertion $S$ such that $S$'s largest letter is $v-1$. The $v$-consistent 4-tuple-letter set $W\vert _{\v=v}={W_{(v,1)}, \dots, W_{(v,k)}}$ is constructible on $S$ if and only if:
\begin{align*}
    \r(W_{(v,1)}) \leq \#v\eob(S) - \mid \Destop(W\vert _{\v=v})\mid
\end{align*}
\end{Lem}
\begin{proof}
First we will try to create valid $v_{k}$-insertion by inserting $W_{(v,k)}$ into $S$. Considering 2 cases:
\begin{itemize}
    \item If $W_{(v,k)}$ is non-descent-top, then
        \begin{align*}
        W_{(v,k)} \text{ is constructible}
            &\iff \text{exists }(\r(W_{(v,k)}))^{th} \text{$v$-embraceable open block from the right}\\
            &\iff \#v\eob(S) \geq \r(W_{(v,k)})
        \end{align*}
    \item If $W_{(v,k)}$ is descent-top, then
        \begin{align*}
            W_{(v,k)} \text{ is constructible} 
            &\iff \text{exists }(\r(W_{(v,k)}) + 1)^{th} \text{$v$-embraceable block from the right}\\
            &\iff \#v\eob(S) \geq \r(W_{(v,k)}) + 1
        \end{align*}
\end{itemize}
Combine 2 cases above, we have
\begin{multline}
    W_{(v,k)} \text{ is constructible} \iff \#v\eob(S) \geq \r(W_{(v,k)}) + \mid \Destop(\{W_{(v,k)}\})\mid
\end{multline}

Next we consider case when creating valid $v_{k-1}$-insertion by inserting $W_{(v,k)}$ into $v_{k}$-insertion.

\noindent
If $W_{v_{k-1}}$ is non-descent-top, then
\begin{align*}
    W_{(v,k-1)} \text{ is constructible}
    &\iff \begin{cases}
    \text{exists } \r(W_{(v,k)})^{th} \text{$v$-embraceable block from the right} \quad \numbAtBack{1}\\
    \text{$v_{k-1}$-insertion is valid} \quad \numbAtBack{2}
    \end{cases}
\end{align*}
We notice that with the assumption $W\vert _{\v=v}$ is $v$-consistent 4-tuple-letter set and $v_{k}$-insertion is valid, (2) is the result of (1). The proof for this was mentioned in the proof for case ($\zeta-2$) and case ($\zeta-3$) in \cref{lma13}.
\begin{align*}
    \qquad \qquad \qquad W_{(v,k-1)} \text{ is constructible}
    &\iff (1)\\
    &\iff \#v\eob(\text{$v_{k}$-insertion}) \geq \r(W_{v_{k-1}})\\
    {\text{Apply \cref{lma17}}\qquad \qquad \qquad \qquad}\\
    &\iff \#v\eob(S) - \mid \Destop(\{W_{(v,k)}\})\mid \geq \r(W_{v_{k-1}}) \qquad
\end{align*}

If $W_{v_{k-1}}$ is descent-top, then
\begin{align*}
    W_{(v,k-1)} \text{ is constructible}
    &\iff \begin{cases}
    \exists \r(W_{(v,k)})^{th} \text{$v$-embraceable open block from the right} \quad \numbAtBack{3}\\
    \text{$v_{k-1}$-insertion is valid} \quad \numbAtBack{4}\\
    \end{cases}
\end{align*}
Similar as the case above, (4) is the result of (3).
\begin{align*}
    W_{(v,k-1)} \text{ is constructible}
    &\iff (3)\\
    &\iff \#v\eob(\text{$v_{k}$-insertion}) \geq \r(W_{v_{k-1}}) + 1\\
    &\iff \#v\eob(S) - \mid \Destop(\{W_{(v,k)}\})\mid \geq \r(W_{v_{k-1}}) + 1
\end{align*}
Combine 2 cases above, we have
\begin{align*}
    W_{(v,k-1)} \text{ is constructible} \iff &\#v\eob(S) \\
    &\geq \r(W_{v_{k-1}}) + \mid \Destop(\{W_{v_{k-1}}\})\mid + \mid \Destop(\{W_{(v,k)}\})\mid\\
    & = \r(W_{v_{k-1}}) + \mid \Destop(\{W_{(v,k)}, W_{v_{k-1}}\})\mid
\end{align*}

Similarly, we can construct the equivalent condition for $W_{(v,i)}$ is constructible on $v_{i+1}$-insertion $\forall i \in \overline{1,k-1}$.

\noindent
According to the definition
\begin{align}
\nonumber
W\vert _{\v=v} \text{ is constructible on } S
& \iff \begin{cases}
\text{$W_{(v,k)}$ is constructible on $S$}\\
\text{$W_{v_{k-1}}$ is constructible on $v_{k}$-insertion}\\
\vdots\\
\text{$W_{(v,1)}$ is constructible on $v_2$-insertion}
\end{cases}\\
\label{elma162}
& \iff \begin{cases}
\#v\eob(S) \geq \r(W_{(v,k)})^{th} + \mid \Destop(\{W_{(v,k)}\})\mid\\
\#v\eob(S) \geq  \r(W_{v_{k-1}}) + \mid \Destop(\{W_{(v,k)}, W_{v_{k-1}}\})\mid\\
\vdots\\
\#v\eob(S) \geq  \r(W_{(v,1)}) + \mid \Destop(W\vert _{\v=v})\mid\\
\end{cases}
\end{align}
Since $\r(W_{(v,1)}) = max(\r(W_{(v,i)}))$
\begin{align*}
\cref{elma162} \iff \#v\eob(S) \geq  \r(W_{(v,1)}) + \mid \Destop(W\vert _{\v=v})\mid
\end{align*}
\end{proof}

\subsection{Function $\epsilon$}
Before defining $\epsilon$, we have following remarks:
\begin{Rem}\label{rem6}
We first consider a simple word $w=3_121-3_2-3_3$. Through the experimental test, the only possible word $w^{\prime}$ satisfied $((2\underline{31}), (\underline{31}2), (\underline{21}1), (1\underline{21}), (\underline{21}2), (2\underline{21}), \underline{21}) w = ((2\underline{31}), (\underline{31}2), (\underline{21}1), (1\underline{21}), ,(\underline{21}2), \underline{21}) w^{\prime}$ is $3_1-3_2-3_321$. We notice 2 things:
\begin{itemize}
    \item $w^{\prime}(3_1)$ of has the descent top status of $w(3_3)$, and the place of $3_3$ among letters $3$ is "opposite" to $3_1$'s place
    \item $w^{\prime}(3_1)$ of has the descent bottom status of $w(3_1)$
\end{itemize}
\end{Rem}
\begin{Rem}\label{rem7}
The occurrences of pattern $(\underline{21}2)\vert _{2=3}$ and $(2\underline{21})\vert _{2=3}$ in word $w$ depends only on
\begin{itemize}
    \item the descent top statuses of letters $3$ 
    \item the place of letters $3$ that are descent tops within all the letters $3$ 
\end{itemize}
For more precisely, $(2\underline{21})\vert _{2=3} = \mid \{(i,j):w_i = w_j = 3, w_j\text{ is descent top}, i < j\}\mid $, while $(\underline{21}2)\vert _{2=3} = \mid \{(i,j):w_i = w_j = 3, w_i\text{ is descent top}, i < j\}\mid $. We can see that the $\Phi$ and $\epsilon$ function must somehow "reverse" the relative places between $3_i$ and $3_j$, one of which is descent top. 
\end{Rem}
\begin{Rem}\label{rem8}
The occurrences of pattern $(\underline{21}1)\vert _{1=3}$ in word $w$ depends only on
\begin{itemize}
    \item the descent bottom statuses of letters $3$ 
    \item the place of letters $3$ that are descent bottoms within all the letters $3$ 
\end{itemize}
$(\underline{21}1)\vert _{1=3} = \mid \{(i,j):w_i = w_j = 3, w_i\text{ is descent bottom}, i < j\}\mid $. The $\Phi$ and $\epsilon$ function must "preserve" the relative places between $3_i$ and $3_j$, one of which is descent bottoms. 
\end{Rem}

\paragraph{$\epsilon$ function}
Given a word $w=w_1w_2\dots w_n$ in $\Sset_{(A,m)}$ and its set of 4-tuple-letter $W=\{W_1,W_2,\dots ,W_n\}$. $\epsilon$ maps the set $\{W_1,\dots ,W_n\}$ to the set $\{U_1,\dots ,U_n\}$ such that $U_i=\theta_W(W_i)$.
\begin{equation*}
    \epsilon:\{W_1,\dots ,W_n\} \mapsto \{\theta_W(W_1),\dots ,\theta_W(W_n)\}
\end{equation*}
The $\epsilon$ function has a domain is $L_{\Sset_{(A,m)}}$. In the \cref{lma20}, we will prove that $L_{\Sset_{(A,m)}}$ is also its codomain. According to the observation from the \cref{rem6} and \cref{rem7}, we will construct functions $\theta_W$ such that:
\begin{align*}
    \text{descent top status of }\theta_W(W_{(v,i)}) 
    &= \text{descent top status of }f_W(W_{(v,i)})\\
    &= \text{descent top status of }W_{v_{k_v + 1 - i}}
\end{align*}
And according to the \cref{rem6} and \cref{rem8}
\begin{align*}
    \text{descent bottom status of }\theta_W(W_{(v,i)}) 
    &= \text{descent bottom status of }(W_{(v,i)})
\end{align*}

\noindent
Function $\theta_W: W_i \mapsto U_i$ is defined as follow:
\begin{align*}
    U_i \begin{cases}
    \v(U_i) = \v(W_i)\\
    \r(U_i) = \r(W_i)\\
    \p(U_i) = h(\p(f_W(W_i)),\p(W_i))\\
    \d(U_i) = \d(W_i)
    \end{cases}
\end{align*}
From the way we define $\epsilon$ and mapping value and duplicate index in $\theta_W$ from $W_i$ to $U_i$, the unique of the tuple $(\v,\d)$ within the set $\{(\v(X),\d(X):X\in U\}$ is still preserved.

\noindent
It is clear that $\theta_W(W_{(v,i)})$ has the same value, right embracing number and duplicate index as $W_i$. Later on in \cref{lma2}, we will also point out that $\theta_W(W_{(v,i)})$ has the same descent bottoms status as $W_i$ and the same descent top status as $f(W_i)$ and $W_{v_{k_v + 1 - i}}$

\subsubsection{$f_W$ function}
\begin{alignat*}{4}
& f_W:   \quad && \{W_1,W_2,\dots,W_n\}  \quad && \to \quad && \{W_1,W_2,\dots,W_n\}\\
&    \quad && \quad \quad \quad W_i \quad && \mapsto \quad && \quad \quad \quad W^{\prime}_i
\end{alignat*}
The $f_W$ function is a piecewise function with sub function $f_{W,v}$ applies to $W_i$ of letter that has value $v$.

\paragraph{$f_{W,v}$ function}
\noindent
Given a consistent 4-tuple-letter set $v$ $W\vert _{\v=v}=\{W_{(v,1)},W_{(v,2)},\dots W_{(v,k)}\}$. The $f_{W,v}$ is presented as follow:
\begin{alignat*}{4}
& f_{W,v}:   \quad && W\vert _{\v=v}  \quad && \to \quad && W\vert _{\v=v}\\
& \quad && W_{(v,j)} \quad && \mapsto \quad &&  W^{\prime}_{(v,j)}\\
& \quad && T_{v_e}(j) \quad && \mapsto \quad &&  T_{v_{\p}}(j)
\end{alignat*}
We present the mapping of $f_{W,v}$ through 2 tuples $T_{v_e}$ and $T_{v_{\p}}$ such that  $f_{W,v}(T_{v_e}(j)) = T_{v_{\p}}(j)$. $T_{v_e} = (W_{(v,1)},W_{(v,2)},\dots W_{(v,k)})$ is a tuple of all 4-tuple-letters in $W\vert _{\v=v}$ with its elements ordered to follow the increase order of duplicate index. Let $(\r(W_{(v,1)}),\dots ,\r(W_{(v,k)}))$ and $(x_{(v,1)},\dots x_{(v,k)})$ be a tuple of right embracing sum and tuple of non-descent-top of $T_{v_e}$, respectively.

\noindent
In order to construct $T_{v_{\p}}$:
\begin{itemize}
    \item We create 2 queues: a queue for non-descent-top $S_{\text{non descent top}}$ and a queue for descent-top $S_{\text{descent top}}$. We start with the first element from $T_{v_e}$. If the position of that element is opener or outsider, then it shall be pushed into stack $S_{\text{non descent top}}$, $S_{\text{descent top}}$ otherwise. We continue with the second element of tuple $T_{v_e}$, to the last element of the tuple.
    \item We construct the tuple $(x^{\prime}_{(v,1)},\dots x^{\prime}_{(v,k)})$ such that $x^{\prime}_{(v,i)} = x_{(v,k+1-i)}$. We denote this mapping is $\omega$
    \item We construct $T_{v_{\p}} = (W^{\prime}_{(v,1)},W^{\prime}_{(v,2)},\dots W^{\prime}_{(v,k)})$, in which
    \begin{align*}
        W^{\prime}_{(v,i)} = 
    \begin{cases}
    S_{\text{descent top}}.\text{pop()} & \quad {\text{if } x^{\prime}_{(v,i)} = 0}\\
    S_{\text{non descent top}}.\text{pop()} & \quad \text{if } x^{\prime}_{(v,i)} = 1
    \end{cases} , \qquad i:0\to k
    \end{align*}
\end{itemize}

\subsubsection{$h$ function}
The $h$ function defines the new position from 2 given position
\begin{alignat*}{4}
& h:   \quad && \{\text{opener}, \text{closer}, \text{insider}, \text{outsider}\}^2  \quad && \to \quad && \{\text{opener}, \text{closer}, \text{insider}, \text{outsider}\}
\end{alignat*}
Follow the \cref{rem9}, it is clear that the position is determined by descent bottoms status and descent top status; and from the position of the letter, we can determined its descent bottoms status and descent top status. Hence, it is possible to construct $h(posX,posY)$ as follow:
\begin{align*}
    h(posX,posY)
    \begin{cases}
    \text{is descent bottom} \iff posX\text{ is descent bottom}\\
    \text{is descent top} \iff posY\text{ is descent top}
    \end{cases}
\end{align*}

Bellow is the details value table for $h(\p(W_i),\p(f_W(W_i)))$
\setlength{\extrarowheight}{2pt}
\begin{center}
    \begin{NiceTabular}{cccccc}[hvlines=NW]
& & \Block{1-4}{position of $f_W(W_i)$} \\ 
& & outsider & opener & insider& closer\\ 
\Block{4-1}<\rotate>{position of $W_i$}
& outsider & outsider & outsider & closer & closer \\ 
& opener & opener & opener & insider & insider\\ 
& insider & opener & opener & insider & insider\\ 
& closer & outsider & outsider & closer & closer \\ 
    \end{NiceTabular}
\end{center}

\begin{Exam}\label{ex2}
Given a consistent set of 4-tuple-letter $W = \{(4,1,\close,0),(3,1,\open,1)$ $,(4,2,\close,0),(2,1,\insi,0),(1,1,\open,0),(2,2,\out,0),(4,3,\out,0)\}$. We will apply the $\epsilon$ function to the set of $W$

\noindent
First, we shall construct function $f_W$:
\begin{itemize}
    \item Consider $f_{W,4}$
    
    \sloppy $W\vert _{\v=4} = \{(4,1,\close,0), (4,2,\close,0), (4,3,\out,0)\}$. $T_{4_e} = ((4,1,\close,0),$ $ (4,2,\close,0), (4,3,\out,0)) $ has tuple of right embracing number $(0,0,0)$ and a tuple of non-descent-top is $(0,0,1)$. The queue for non-descent-top $S_{\text{non descent top}}$ and queue for descent-top $S_{\text{descent top}}$ has tuple representation is $((4,3,\out,0))$ and $((4,1,\close,0),$ $(4,2,\close,0))$, respectively. The tuple of non-descent-top tuple of $T_{4_e}$ implies that $(x^{\prime}_{(v,1)},\dots x^{\prime}_{(v,k)})$ equals $(1,0,0)$. Consequently, $T_{v_{\p}} = ((4,3,\out,0),(4,1,\close,0),(4,2,\close,0))$.
    
    We conclude that
    \begin{alignat*}{7}
    & f_W: && (4,1,\close,0) \quad && \mapsto \quad && (4,3,\out,0) \qquad && (4,2,\close,0) \quad && \mapsto \quad && (4,1,\close,0)\\
    &  && (4,3,\out,0) \quad && \mapsto \quad && (4,2,\close,0) \qquad &&   \quad &&   \quad &&  
    \end{alignat*}
    \item Similarly, with with input $W_i$ having value equals to 1,2,3:
    \begin{alignat*}{7}
    & f_W: && (1,1,\open,0) \quad && \mapsto \quad && (1,1,\open,0) \qquad &&   \quad &&   \quad &&  \\
    & && (2,1,\insi,0) \quad && \mapsto \quad && (2,2,\out,0) \qquad && (2,2,\out,0) \quad && \mapsto \quad && (2,1,\insi,0)\\
    & && (3,1,\open,1) \quad && \mapsto \quad && (3,1,\open,1) \qquad &&   \quad &&   \quad &&  
    \end{alignat*}
\end{itemize}
Next, we will construct $U$:
\begin{center}
\begin{tabular}{ |c|c|c|c|c|c|c| } 
 \hline
 $W_i$ & $f_W(W_i)$ & \vtop{\hbox{\strut $\v(U_i)$}\hbox{\strut $=\v(W_i)$}} & \vtop{\hbox{\strut $\r(U_i)$}\hbox{\strut $=\r(W_i)$}} & \vtop{\hbox{\strut $\p(U_i)$}\hbox{\strut $=h(\p(f(W_i)),\p(W_i))$}} & \vtop{\hbox{\strut $\d(U_i)$}\hbox{\strut $=\d(W_i)$}}\\
 \hline
 $(4,1,\close,0)$ & $(4,3,\out,0)$ & 4 & 0 & $\out$ & 1 \\ 
 $(4,2,\close,0)$ & $(4,1,\close,0)$ & 4 & 0 & $\close$ & 2 \\ 
 $(4,3,\out,0)$ & $(4,2,\close,0)$ & 4 & 0 & $\close$ & 3 \\
 $(1,1,\open,0)$ & $(1,1,\open,0)$ & 1 & 0 & $\open$ & 1 \\ 
 $(2,1,\insi,0)$ & $(2,2,\out,0)$ & 2 & 0 & $\open$ & 1 \\ 
 $(2,2,\out,0)$ & $(2,1,\insi,0)$ & 2 & 0 & $\close$ & 2 \\
 $(3,1,\open,1)$ & $(3,1,\open,1)3$ & 3 & 1 & $\open$ & 1 \\
 \hline
\end{tabular}
\end{center}
So 
\begin{multline*}
    \epsilon(W) = U = \{(4,1,\out,0),(4,2,\close,0),(4,3,\close,0),(1,1,\open,0),\\(2,1,\open,0),(2,2,\close,0),(3,1,\open,1)\}
\end{multline*}
\end{Exam}

\begin{Lem}\label{lma18}
$f_{W,v}$ is an involution
\end{Lem}
\begin{proof}
Firstly, from the way we construct $(W^{\prime}_{(v,1)},W^{\prime}_{(v,2)},\dots W^{\prime}_{(v,k)})$ based on the $(x^{\prime}_{(v,1)},\dots x^{\prime}_{(v,k)})$ tuple, it is clear that $(x^{\prime}_{(v,1)},\dots x^{\prime}_{(v,k)})$ is the tuple of non-descent top of $(W^{\prime}_{(v,1)},W^{\prime}_{(v,1)},\dots W^{\prime}_{(v,k)})$ \numbAtBack{1}\\
Secondly, the mapping (let's call this $f_W^{\prime}$ from $(x_{(v,1)},\dots x_{(v,k)})$ to $(x^{\prime}_{(v,1)},\dots x^{\prime}_{(v,k)})$ is the involution. \numbAtBack{2}\\
Thirdly, $(W_{(v,1)},W_{(v,1)},\dots W_{(v,k)})$ and $(W^{\prime}_{(v,1)},W^{\prime}_{(v,2)},\dots W^{\prime}_{(v,k)})$ shared the same queue for non-descent-top $S_{\text{non descent top}}$ and queue for descent-top $S_{\text{descent top}}$. \numbAtBack{3}

Assume that we run $f_{W,v}$ on the input $(W^{\prime}_{(v,1)},W^{\prime}_{(v,2)},\dots W^{\prime}_{(v,k)})$, since (1) and (2), the binary tuple for constructing output $(x^{\dprime}_{(v,1)},\dots x^{\dprime}_{(v,k)})$ equals to $(x_{(v,1)},\dots x_{(v,k)})$ (4). Since (3) and (4), the tuple $T^{\prime}_{v_{\p}} = (W^{\dprime}_{(v,1)},W^{\dprime}_{(v,2)},\dots W^{\dprime}_{(v,k)})$ equals to $(W_{(v,1)},W_{(v,2)},\dots W_{(v,k)})$, which implies that $f_{W,v}$ is an involution.
\end{proof}
From the \cref{lma18}, it is clear that
\begin{Lem}\label{lma1}
$f_W$ is an involution
\end{Lem}

\begin{Lem}\label{lma2}
Given $\p(B)=h(\p(A),\p(C))$. Then $B$ is a descent top if and only if $C$ is a descent top; and $B$ is a descent bottom if and only if $A$ is a descent bottom
\end{Lem}
\begin{proof}
$X$ is a descent top if and only if $\p(X)$ is insider or closer. From the $h$ mapping table, we will construct the new table to see how the descent top status of $A$ and $C$ affect to $B$ by replacing "opener" and "outsider" with "non-descent-top", and replacing "insider" and "closer" with "descent-top"
\begin{center}
    \begin{NiceTabular}{cccccc}[hvlines=NW]
& & \Block{1-4}{position of $C$} \\ 
& & non-descent-top & non-descent-top & descent-top & descent-top\\ 
\Block{4-1}<\rotate>{position of $A$}
& non-descent-top & non-descent-top & non-descent-top & descent-top & descent-top \\ 
& non-descent-top & non-descent-top & non-descent-top & descent-top & descent-top\\ 
& descent-top & non-descent-top & non-descent-top & descent-top & descent-top\\ 
& descent-top & non-descent-top & non-descent-top & descent-top & descent-top \\ 
    \end{NiceTabular}
\end{center}
From the table above, we can see that with $\p(B)=h(\p(A),\p(C))$, then $B$ is a descent top if and only if $C$ is a descent top.\par
Similarly, we construct the table for descent-bottom status by replacing "closer" and "outsider" with "non-descent-bottom", and replacing "insider" and "opener" with "descent-bottom". It implies that $B$ is a descent bottom if and only if $A$ is a descent bottom.
\begin{center}
    \begin{NiceTabular}{cccccc}[hvlines=NW]
& & \Block{1-4}{position of $C$} \\ 
& & non-descent-bot & descent-bot & descent-bot & non-descent-bot\\ 
\Block{4-1}<\rotate>{position of $A$}
& non-descent-bot & non-descent-bot & outsider & non-descent-bot & non-descent-bot\\ 
& descent-bot & descent-bot & descent-bot & descent-bot & descent-bot\\ 
& descent-bot & descent-bot & descent-bot & descent-bot & descent-bot\\ 
& non-descent-bot & non-descent-bot & non-descent-bot & non-descent-bot & non-descent-bot \\ 
    \end{NiceTabular}
\end{center}
\end{proof}

\begin{Lem}\label{lma3}
Given a set of 4-tuple-letter $W = \{W_1,W_2,\dots W_n\}$. Let $W\vert _{\v=v}$ be a set of 4-tuple-letter in $W$ that has value is $v$. The $\epsilon$ function mapping $W$ to $U = \{U_1,U_2,\dots U_n\}$. Then:
\begin{align*}
    \Destop(W\vert _{\v=v}) &= \Destop(U\vert _{\v=v})\\
    \Desbot(W\vert _{\v=v}) &= \Desbot(U\vert _{\v=v})\\
    \Destop(W) &= \Destop(U)\\
    \Desbot(W) &= \Desbot(U)
\end{align*}
\end{Lem}
\begin{proof}
According to the definition of $\theta_W$, we have $\p(U_i)=h(\p(W_i),\p(f(W_i))$.\\
According to \cref{lma2}, $W_{(v,i)}$ is a descent top $\iff$ $U_{(v,i)}$ is a descent top. So that $W\vert _{\v=v}$, which is the set of all $W_{(v,i)}$ will has equal number descent top letter as $U\vert _{\v=v}$ which is the set of all $U_{(v,i)}$
\begin{align*}
    \Desbot(W\vert _{\v=v}) = \Desbot(U\vert _{\v=v})
\end{align*}
Consequently
\begin{align*}
    \Desbot(W) = \Desbot(U)
\end{align*}

Let $\Destop(W^{\prime}\vert _{\v=v}) = \{\v(f(W_i)) : \p(f(W_i)) = \text{insider} \text{ or} \text{ closer}\}$. Applying the \cref{lma2}
\begin{equation}
\Destop(W^{\prime}\vert _{\v=v}) = \Destop(U\vert _{\v=v})
\end{equation}
Since $f_{W,v}$ is the involution, followed \cref{lma1}
\begin{equation}
\Destop(W^{\prime}\vert _{\v=v}) = \Destop(W\vert _{\v=v})
\end{equation}
(6.1) and (6.2) implies that $\Destop(U\vert _{\v=v}) = \Destop(W\vert _{\v=v})$. Consequently, $\Destop(U) = \Destop(W)$
\end{proof}

\begin{Lem}\label{lma19}
Given a consistent 4-tuple-letter set $W = \{W_1,\dots,W_n\}$. Let $U=\{U_1,\dots,U_n\}$ be an image of $W$ under $\epsilon$. Let $U\vert _{\v=v}$ be a set of all letters that has value is $v$ in $U$, respectively. Then $U\vert _{\v=v}$ is the $v$-consistent 4-tuple-letter set.
\end{Lem}
\begin{proof}
Let $W\vert _{\v=v}$ is the set of all letters in $W$ that has the value is $V$. Since the $\theta_W$ remains the value in the mapping, the $W\vert _{\v=v}$ is the set of preimages of $U\vert _{\v=v}$.\\
Follow the \cref{lma21}, $W\vert _{\v=v}$ is the $v$-consistent 4-tuple-letter set. The definition about $v$-consistent 4-tuple-letter set is the set of restrictions of duplicate indices, value, and embracing values between 4-tuple-letters within the set. Since $U_{(v,i)}$ of $U\vert _{\v=v}$ has the same duplicate index, value, and embracing value as $W_{(v,i)}$ of $W\vert _{\v=v}$, it is clear that $U\vert _{\v=v}$ is $v$-consistent 4-tuple-letter set.
\end{proof}

\begin{Lem}\label{lma20}
Given a well-defined 4-tuple-letter set $W$. Let $U$ be an image of $W$ under $\epsilon$. Then $U$ is also well-defined
\end{Lem}
\begin{proof}
Since $W$ is well-defined, exists a word $w$ such that $\delta(w)=W$, with $W=\{W_1,\dots,W_k\}$. Let $U=\{U_1,\dots,U_k\}$ be an image of $W$ under $\epsilon$. We will prove $U$ is well-defined follow the approach mentioned in \cref{rem5}. Two over three necessary conditions are satisfied according to the \cref{lma3} and the \cref{lma19}. The remained condition need to be proved is:
\begin{align}\label{elma2011}
    \forall v \in A, U\vert _{\v=v} \text{ is constructible on }{v-1}_{k_{v-1}}-\text{insertion}
\end{align}

We consider when $v = 1$. Let $S_{W,0}$ and $S_{U,0}$ be insertions with no blocks. According to the definition of the valid insertion, $S_{W,0}$ and $S_{U,0}$ are valid. It is clear that $U\vert _{\v=1}$ and $W\vert _{\v=1}$ is $1$-consistent 4-tuple-letter set.\\
Since $U\vert _{\v=1}$ is already the $1$-consistent 4-tuple-letter set followed \cref{lma19}, if $U\vert _{\v=1}$ is constructible on $S_{W,0}$, it implies that we can construct valid $1_i$-insertions of $U$ under $\zeta$.\\
$W\vert _{\v=1}$ is $1$-consistent 4-tuple-letter set that is constructible on $S_{W,0}$. According to \cref{lma19}, it is equivalent to
\begin{align}\label{elma201}
    \r(W_{1_1}) \leq \#1\eob(S_{W,0} - \mid \Destop(W\vert _{\v=1})\mid
\end{align}
According to $\theta$ mapping:
\begin{align}\label{elma202}
    \r(W_{1_1}) = \r(U_{1_1})
\end{align}
Since both $S_{W,0}$ and $S_{U,0}$ is the insertion with no blocks:
\begin{align}\label{elma203}
    \#1\eob(S_{W,0}) = \#1\eob(S_{U,0})
\end{align}
Follow the \cref{lma3}
\begin{align}\label{elma204}
    \Destop(W\vert _{\v=1}) = \Destop(U\vert _{\v=1})
\end{align}
Apply \cref{elma202}, \cref{elma203} and \cref{elma204} into \cref{elma201}
\begin{align}\label{elma205}
    \r(U_{1_1}) \leq \#1\eob(S_{U,0}) - \mid \Destop(U\vert _{\v=1})\mid
\end{align}
Since $U\vert _{\v=1}$ is the $1$-consistent 4-tuple-letter set, according to \cref{lma19}, \cref{elma205}, is equivalent to that $U\vert _{\v=1}$ is constructible on $S_{W,0}$. Hence we can construct valid $1_i$-insertions of $U$ under $\zeta$.

Assume that for all $l\in [v-1]$, the $l$ consistent 4-tuple-letter set $U\vert _{\v=l}$ is constructible on ${l-1}_{k_{l-1}}$-insertion of $U$, which means we can construct valid $l_{i_l}$-insertions of $U$ under $\zeta$. We shall prove that $U\vert _{\v=v}$ is constructible on $v-1_{k_{v-1}}$-insertion of $U$.

\noindent
Let $S_{W,v-1}$ and $S_{U,v-1}$ be $v-1_{k_{v-1}}$-insertions of $W$ and the one of $U$ respectively. Since $W$ is a well-defined unit letter set, $S_{W,v-1}$ is valid. According to the induction assumption, $S_{U,v-1}$ is also valid.\\
It is clear that $U\vert _{\v=v}$ and $W\vert _{\v=v}$ is $1$-consistent 4-tuple-letter set. $W\vert _{\v=v}$ is $v$-consistent 4-tuple-letter set that is constructible on $S_{W,0}$, which according to \cref{lma19}, is equivalent to
\begin{align}\label{elma206}
    \r(W_{(v,1)}) \leq \#v\eob(S_{W,v-1}) - \mid \Destop(W\vert _{\v=v})\mid
\end{align}
According to $\theta$ mapping
\begin{align}\label{elma207}
    \r(W_{(v,1)}) = \r(U_{(v,1)})
\end{align}
Now we consider about the number of $v$-embraceable open block in $S_{W,v-1}$. It is clear that:
\begin{align*}
    \text{$\#$open block} 
    &= \text{$\#$opener} - \text{$\#$closer}\\
    &= (\text{$\#$opener} + \text{$\#$insider})
    - (\text{$\#$closer} + \text{$\#$insider})\\
    &= \text{$\#$descent top} - \text{$\#$descent bottom}\\
\end{align*}
Since $S_{W,v-1}$ does not contain letter $v$, all the open blocks in $S_{W,v-1}$ are $v$-embraceable.
\begin{align*}
    \#v\eob(S_{W,v-1})
    &= \text{$\#$descent top in $S_{W,v-1}$} - \text{$\#$descent bottom in $S_{W,v-1}$}\\
    &= \mid \Destop(W\vert _{v \in \overline{1,v-1}})\mid - \mid \Desbot(W\vert _{v \in \overline{1,v-1}})\mid
\end{align*}
Similarly
\begin{align*}
    \#v\eob(S_{U,v-1}) = \mid \Destop(U\vert _{v \in \overline{1,v-1}})\mid - \mid \Desbot(U\vert _{v \in \overline{1,v-1}})\mid
\end{align*}
Follow the \cref{lma3}, $\Destop(W\vert _{v \in \overline{1,v-1}}) = \Destop(U\vert _{v \in \overline{1,v-1}})$ and $\Desbot(W\vert _{v \in \overline{1,v-1}}) = \Desbot(U\vert _{v \in \overline{1,v-1}})$. Hence
\begin{align}\label{elma208}
    \#v\eob(S_{W,v-1}) = \#v\eob(S_{U,v-1})
\end{align}
Follow the \cref{lma3}
\begin{align}\label{elma209}
    \Destop(W\vert _{\v=v}) = \Destop(U\vert _{\v=v})
\end{align}
Apply \cref{elma207}, \cref{elma208} and \cref{elma209} into \cref{elma206}
\begin{align}\label{elma2010}
    \r(U_{(v,1)}) \leq \#v\eob(S_{U,v-1}) - \mid \Destop(U\vert _{\v=v})\mid
\end{align}
Since $U\vert _{\v=v}$ is the $v$-consistent 4-tuple-letter set, according to \cref{lma19}, \cref{elma2010} is equivalent to that $U\vert _{\v=1}$ is constructible on $S_{W,v-1}$. Hence we can construct valid $v_i$-insertions of $U$ under $\zeta$.

Following induction proof, \cref{elma2011} is proved. The $U$ is a well-defined 4-tuple-letter set
\end{proof}

\begin{Lem}\label{lma4}
Given a word $w$. Let $\Destop(W)$ and $\Desbot(W)$ be a multiset of descent top and a multiset of descent bottom of $w$, respectively. Let $v$ be a letter in $w$. Then the number of descent blocks that can embrace $v$ in $w$ is determined by $\Destop(W)$ and $\Desbot(W)$.
\end{Lem}
\begin{proof}
Based on the definition of descent top and descent bottom, it is clear that the number of descent top and the number of descent bottom in the word equals each other.
\begin{align*}
    \mid \Destop(W)\mid = \mid \Desbot(W)\mid
\end{align*}
Let $T_1$ be a multiset of letter in $\Destop(W)$ that is greater than $v$, $T_2$ be a multiset of letter in $\Destop(W)$ that equals to $v$, and $T_3$ be a multiset of letter in $\Destop(W)$ that is smaller than $v$. Let $t_i$ be a number of element in $T_i$.

\noindent
Respectively, we define $U_1$, $U_2$, $U_3$ and $b_i$ for multiset $\Desbot(W)$.

Let $P=\{(t,b):t\in \Destop(W), b\in \Desbot(W)\}$ be a multiset of all descent block in $w$. We define $P_i$ with $i\in[5]$ as follow:
\begin{itemize}
    \item case when both descent top and descent bottom greater than $v$\\
    $P_1=\{(t,b):t\in T_1, b\in U_1\}$
    \item case when both descent top greater than $v$ and descent bottom equals to $v$\\
    $P_2=\{(t,b):t\in T_1, b\in U_2\}$
    \item case when both descent top equals to $v$ and descent bottom less than $v$\\
    $P_3=\{(t,b):t\in T_2, b\in U_3\}$
    \item case when both descent top and descent bottom less than $v$\\
    $P_4=\{(t,b):t\in T_3, b\in U_3\}$
    \item case when both descent top greater than $v$ and descent bottom less than $v$\\
    $P_5=\{(t,b):t\in T_1, b\in U_3\}$
\end{itemize}
It is clear that $P=P_1\cup P_2\cup P_3\cup P_4\cup P_5$ and $P_i\cap P_j = \emptyset$. Let $p_i = \mid P_i\mid $.

We construct the equality between $t_i$, $b_i$ and $p_i$. We start with $t_1$.
\begin{align*}
    T_1 &= \{t:t \text{ is the first element of }(t,b)\in P_1\}\\
    &\qquad \cup \{t:t \text{ is the first element of }(t,b)\in P_2\}\\
    &\qquad \cup \{t:t \text{ is the first element of }(t,b)\in P_5\}
\end{align*}
and
\begin{align*}
    \{t:t \text{ is the first element of }(t,b)\in P_1\} \cap \{t:t \text{ is the first element of }(t,b)\in P_2\} = \emptyset\\
    \{t:t \text{ is the first element of }(t,b)\in P_1\} \cap \{t:t \text{ is the first element of }(t,b)\in P_5\} = \emptyset\\
    \{t:t \text{ is the first element of }(t,b)\in P_5\} \cap \{t:t \text{ is the first element of }(t,b)\in P_2\} = \emptyset
\end{align*}
Then $t_1 = p_1 + p_2 + p_5$.\\
Respectively, we have $b_2=p_2$ and $b_1=p_1$. Then $p_5 = t_1-b_2-b_1$. 

We notice that $t_i$ and $b_i$ are determined by $\Destop(W)$ and $\Desbot(W)$. The multiset $P_5$ is also the multiset of all the descent blocks that embrace $v$. All of these imply that the number of descent blocks that can embrace $v$ in $w$ is determined by $\Destop(W)$ and $\Desbot(W)$.
\end{proof}

\begin{Rem}\label{rem1}
Given a word $w$. Let $\epsilon(W) = U$ and $\epsilon(\epsilon(W))) = U^{\prime}$. As remarked in \cref{rem2}, with every $A$ in $W$, the tuple $(\v(A),\d(A))$ is unique within $W$ elements. Since $\theta_W$ preserves $\r$, $\d$, $\v$ of $A \in W$, there exists one only $B \in U$ has equals tuple $(\v,\d)$ to $A$, which is the image of $A$ under $\theta_W$, hence we can denote it as $U_i$. Consequently, there exists one only $C \in U^{\prime}$ has equals tuple $(\v,\d)$ to $A$, which is the image of $U_i$ under $\theta_{U}$. We denote this as $C_i$.
\begin{equation*}
    C_i = \theta_{\epsilon(W)}(\theta_W(W_i))
\end{equation*}
\end{Rem}

\begin{Lem}\label{lma12}
Given a word $w$ and the 4-tuple-letter $W_i$ is a 4-tuple-letter in $w$. Let $\epsilon(W)= U$ and $\epsilon(\epsilon(W))) = W^{\dprime}$. Let $U_i$ and $C_i$ is the 4-tuple-letter in $U$ and $W^{\dprime}$ that has the same $(\v,\d)$ to $W_i$. Then $\p(C_i) = \p(W_i)$
\end{Lem}
\begin{proof}
As we remarked in the definition of position, the position of the letter depends on whether it is the descent bot and whether it is descent top. We define the mapping \begin{align*}
    g_{\p}(bot, top) = 
    \begin{cases}
    \text{opener} & \quad {\text{if } bot = true , top = false}\\
    \text{closer} & \quad {\text{if } bot = false , top = true}\\
    \text{insider} & \quad {\text{if } bot = true , top = true}\\
    \text{outsider} & \quad {\text{if } bot = false , top = false}
    \end{cases}
\end{align*}
It is clear that $g_{\p}$ is a bijection.

\noindent
We define $g_{bot}(X)$ and $g_{top}(X)$ returns whether $X$ is a descent bottoms and descent top.

The position of the 4-tuple-letter $X$
\begin{equation}
    \p(W_i) = g_{\p}(g_{bot}(W_i),g_{top}(W_i))
\end{equation}
We have:
\begin{align}\label{elma121}
    \p(C_i) = g_{\p}(g_{bot}(C_i),g_{top}(C_i)
\end{align}

As proved in the \cref{lma2}, $\theta_W(X)$ is a descent bot if and only if $X$ is a descent bot, which means
\begin{equation}\label{elma122}
    g_{bot}(C_i) = g_{bot}(U_i) = g_{bot}(W_i)
\end{equation}

Assume that $W_i$ has value is $v$ and duplicate index is $i$, we denote $W_i$ as $W_{(v,i)}$. According to the way we construct $\theta$, then $U_i$ and $C_i$ also has the same value and duplicate index as $W_i$.\\
Assume in the word $w$ has $k_v$ duplicate letter with value $v$. According to the mapping $f_W$, then $f_W(W_{v_{k_v+1-i}})$ will have the same non-descent-top status to $W_{(v,i)}$, which means
\begin{equation}\label{elma123}
    g_{top}(W_{(v,i)}) = g_{top}(f(W_{v_{k_v+1-i}}))
\end{equation}
Follow the \cref{lma2}, then:
\begin{equation}\label{elma124}
    g_{top}(f(W_{v_{k_v+1-i}})) = g_{top}(U_{v_{k_v+1-i}})
\end{equation}
We replace $W_i$ in \cref{elma123} with $U_{v_{k_v+1-i}}$
\begin{align}\label{elma125}
    \nonumber
    g_{top}(U_{v_{k_v+1-i}}) &= g_{top}(f(U_{v_{k_v+1-(k_v+1-i)}}))\\
    \nonumber
    &= g_{top}(f(U_{(v,i)}))\\
    &= g_{top}(C_{v_i})
\end{align}
From \cref{elma123}, \cref{elma124} and \cref{elma125}:
\begin{equation}\label{elma126}
    g_{top}(W_i) = g_{top}(C_i)
\end{equation}
Apply \cref{elma122} and \cref{elma126} into \cref{elma121}:
\begin{equation}
    \p(C_i) = g_{\p}(g_{bot}(W_i),g_{top}(W_i)) = \p(W_i)
\end{equation}
\end{proof}

From the \cref{lma20}, it is clear that $\epsilon$ has the same domain and codomain, which is $L_{\Sset_{(A,m)}}$.
\begin{align*}
    \epsilon : L_{\Sset_{(A,m)}} \to L_{\Sset_{(A,m)}}
\end{align*}
We will prove the $\epsilon$ is bijective by proving it is an involution.
\begin{Lem}\label{lma9}
$\epsilon$ is an involution
\end{Lem}
\begin{proof}
As remarked in \cref{rem1}, with each $W_i \in W$, exist only one $C_i \in U^{\prime}$ has the same value, duplicate index, and right embracing sum as $W_i$.
And according to \cref{lma12}, $\p(C_i) = \p(W_i)$. Then $C_i=W_i$. Follows that, $\epsilon$ is an involution
\end{proof}

The follow theory is the result of \cref{lma9}
\begin{The}\label{the3}
$\epsilon$ is bijective
\end{The}

\vspace{5mm}
\begin{The}\label{the5}
$\Phi$ is an involution
\end{The}
\begin{proof}
The theory is obvious because:
\begin{itemize}
    \item $\Phi(w) = \zeta \circ \epsilon \circ \delta(w)$
    \item $\epsilon$ is an involution, according to \cref{lma9}
    \item both $\delta$ and $\zeta$ are bijective (\cref{the1}, \cref{the2}), and $\delta$ is an inverse of $\zeta$ (\cref{lma24})
\end{itemize}
\end{proof}
At the end of this section, we will consider an example of the $\Phi$ function
\begin{Exam}
Given a word $w=3-54-41-1-42-53$. We will calculate $\Phi(w)$

First, we need to find the 4-tuple-letter set of $w$ follow $\delta$ function.
\[\begin{matrix}
   w & 3_1 & 5_1 & 4_1 & 4_2 & 1_1 & 1_2 & 4_3 & 2_1 & 5_2 & 3_2 \\
     & W_1 & W_2 & W_3 & W_4 & W_5 & W_6 & W_7 & W_8 & W_9 & W_{10}\\
   \v & 3 & 5 & 4 & 4 & 1 & 1 & 4 & 2 & 5 & 3\\
   \d & 1 & 1 & 1 & 2 & 1 & 2 & 3 & 1 & 2 & 2\\
   \p & \text{outsider} & \text{closer} & \text{opener} & \text{closer} & \text{opener} & \text{outsider} & \text{closer} & \text{opener} & \text{closer} & \text{opener}\\
   \r & 2 & 0 & 1 & 1 & 0 & 0 & 1 & 0 & 0 & 0
\end{matrix}
\]
From the table above, we conclude that:
\begin{multline*}
    \delta(w) = W = \{(3,1,\out,2),(5,1,\close,0),(4,1,\open,1),(4,2,\close,1),(1,1,\open,0),\\
    (1,2,\out,0),(4,3,\close,1),(2,1,\open,0),(5,2,\close,0),(3,2,\open,0)\}
\end{multline*}

Next, we will find $U = \epsilon(W)$.
\begin{enumerate}
    \item We construct function $f_W$:
\begin{itemize}
    \item $f_{W,4}$\\
    $W\vert _{\v=4} = \{W_{(v,1)},W_{(v,2)},W_{(v,3)}\} = \{W_3,W_4,W_7\}$.\\
    $T_{4_e} = (W_{(v,1)},W_{(v,2)},W_{(v,3)}) = (W_3,W_4,W_7)$ with tuple of right embracing number is $(1,1,1)$ and tuple of non-descent-top $(1,0,0)$\\
    $S_{\text{non descent top}} = (W_{(4,1)}) = (W_3)$ and $S_{\text{descent top}} = (W_{(4,2)},W_{(4,3)}) = (W_4,W_7)$\\
    The tuple $(x^{\prime}_{(4,1)},\dots x^{\prime}_{(4,3)}) = (0,0,1)$\\
    Then $T_{v_{\p}} = (W^{\prime}_{(v,1)},W^{\prime}_{(v,2)},W^{\prime}_{(v,3)}) = (W^{\prime}_3,W^{\prime}_4,W^{\prime}_7) = (W_{(4,2)},W_{(4,3)},W_{(4,1)}) = (W_4,W_7,W_3)$\par
    We conclude that $f_W: W_3 \mapsto W_4 \qquad W_4 \mapsto W_7 \qquad W_7 \mapsto W_3$
    \item Similarly, with with input $W_i$ that has value equals to $1$,$2$,$3$,$5$:
    \begin{itemize}
        \item $f_W: W_5 \mapsto W_5 \qquad W_6 \mapsto W_6$
        \item $f_W: W_8 \mapsto W_8$
        \item $f_W: W_1 \mapsto W_1 \qquad W_{10} \mapsto W_{10}$
        \item $f_W: W_2 \mapsto W_2 \qquad W_9 \mapsto W_9$
    \end{itemize}
\end{itemize}

    \item We calculate $U_i = \theta_W(W_i)$
\begin{center}
\begin{tabular}{ |c|c|c|c|c|c|c| } 
 \hline
 $U_i$ & $W_i$ & $f_W(W_i)$ & \vtop{\hbox{\strut $\v(U_i)$}\hbox{\strut $=\v(W_i)$}} & \vtop{\hbox{\strut $e(U_i)$}\hbox{\strut $=e(W_i)$}} & \vtop{\hbox{\strut $\p(U_i)$}\hbox{\strut $=h(\p(f(W_i)),\p(W_i))$}} & \vtop{\hbox{\strut $\d(U_i)$}\hbox{\strut $=\d(W_i)$}}\\
 \hline
 $U_1$ & $W_1$ & $W_1$ & 3 & 2 & outsider & 1 \\ 
 $U_2$ & $W_2$ & $W_2$ & 5 & 0 & closer & 1 \\ 
 $U_3$ & $W_3$ & $W_4$ & 4 & 1 & insider & 1 \\ 
 $U_4$ & $W_4$ & $W_7$ & 4 & 1 & closer & 2 \\ 
 $U_5$ & $W_5$ & $W_5$ & 1 & 0 & opener & 1 \\ 
 $U_6$ & $W_6$ & $W_6$ & 1 & 0 & outsider & 2 \\ 
 $U_7$ & $W_7$ & $W_3$ & 4 & 1 & outsider & 3 \\ 
 $U_8$ & $W_8$ & $W_8$ & 2 & 0 & opener & 1 \\ 
 $U_9$ & $W_9$ & $W_9$ & 5 & 0 & closer & 2 \\ 
 $U_{10}$ & $W_{10}$ & $W_{10}$ & 3 & 0 & opener & 2\\ 
 \hline
\end{tabular}
\end{center}
\end{enumerate}
    
Finally, we construct $w^{\prime}$ from $U$ by using $\zeta$
\begin{itemize}
    \item $1_2$-insertion: $1_2$
    \item $1_1$-insertion: $\infty 1_1 - 1_2$
    \item $2_1$-insertion: $\infty 1_1 - 1_2 - \infty 2_1$
    \item $3_2$-insertion: $\infty 1_1 - 1_2 - \infty 2_1 - \infty 3_2$
    \item $3_1$-insertion: $3_1 - \infty 1_1 - 1_2 - \infty 2_1 - \infty 3_2$
    \item $4_3$-insertion: $3_1 - \infty 1_1 - 1_2 - \infty 2_1 - 4_3 - \infty 3_2$
    \item $4_2$-insertion: $3_1 - \infty 1_1 - 1_2 - 4_2 2_1 - 4_3 - \infty 3_2$
    \item $4_1$-insertion: $3_1 - \infty 4_1 1_1 - 1_2 - 4_2 2_1 - 4_1 - \infty 3_2$
    \item $5_2$-insertion: $3_1 - \infty 4_1 1_1 - 1_2 - 4_2 2_1 - 4_3 - 5_2 3_2$
    \item $5_1$-insertion: $3_1 - 5_1 4_1 1_1 - 1_2 - 4_2 2_1 - 4_3 - 5_2 3_2$
\end{itemize}
We conclude that $\Phi(3-54-41-1-42-53) = 3-541-1-42-4-53$
\end{Exam}

\section{Patterns through bijection}
In this section, we will prove that: 
\begin{equation*}
    ((2\underline{31}),(\underline{31}2),(\underline{21}1),(\underline{21}2),(2\underline{21}),\underline{21})w = ((2\underline{31}),(\underline{31}2),(\underline{21}1),(2\underline{21}),(\underline{21}2),\underline{21})\Phi(w)
\end{equation*}

\begin{Lem}\label{lma5}
Given a word $w$. Let $v$ be a letter in $w$. Then the number of descent blocks that can embrace $v$ in $w$ equals to the number of descent blocks that can embrace $v$ in $\Phi(w)$.
\end{Lem}
\begin{proof}
We denote the 4-tuple-letter set of $w$ is $W$ and the 4-tuple-letter set of $\Phi(w)$ is $U$. According to the definition of $\Phi$, $U$ is the image of $W$ under $\epsilon$. Follow the \cref{lma3}:
\begin{align} \label{elma51}
    \Desbot(W) = \Desbot(U)
\end{align}
and
\begin{align} \label{elma52}
    \Destop(W) = \Destop(U)
\end{align}
Follow the \cref{lma4}, the number of descent blocks that can embrace $v$ in $w$ is determined by $\Desbot(W)$ and $\Destop(W)$, and the number of descent blocks that can embrace $v$ in $w^{\prime}$ is determined by $\Desbot(U)$ and $\Destop(U)$. Along with \cref{elma51} and \cref{elma52}, the lemma is proved.
\end{proof}

\begin{Lem}\label{lma6}
Given a word $w=w_1w_2\dots w_n$ in $\Sset_{(A,m)}$ with its set of 4-tuple-letter $W=\{W_1,\dots W_n\}$. Let $U = \{U_1,U_2,\dots ,U_n\}$ in which $U_i$ be an image of $W_i$ under $\theta_W$. Let $w^{\prime} = W^{\prime}_1\dots W^{\prime}_n$ be a word constructed from $U$ under $\zeta$, and $W^{\prime}_j$ is the present of $U_i$ in $w^{\prime}$. Then:
\begin{align*}
    \l(U_i) = \l({W^{\prime}_j}) = \l(W_i) = \l(w_i)
\end{align*}
\end{Lem}
\begin{proof}
Since $W^{\prime}_j$ is the present of $U_i$ in $w^{\prime}$,
\begin{align}
    \v(U_i) &= \v(W^{\prime}_j)\\
    \r(U_i) &= \r(W^{\prime}_j)\\
    \l(U_i) &= \l(W^{\prime}_j)
\end{align}
We notice that in here, the term \textit{embraced} is "strictly embraced", hence the block that can embraced $v$ will be either its right embracing block or left embracing block.
\begin{align*}
    \l(U_i) = e_{l_{W^{\prime}_j}} &=\text{$\#$descent block in $w^{\prime}$ can embrace } \v(W^{\prime}_j) - \r({W^{\prime}_j})\\
    &=\text{$\#$descent block in $w^{\prime}$ can embrace } \v(U_i) - \r(U_i)\\
    &=\text{$\#$descent block in $w$ can embrace } \v(W_i) - \r(W_i)\\
    &=\text{$\#$descent block in $w$ can embrace } \v(w_i) - \r(W_i)\\
    &= \l(W_i) = \l(w_i)
\end{align*}
\end{proof}

\begin{Lem}\label{lma7}
Given a word $w=w_1w_2\dots w_n$ in $\Sset_{(A,m)}$ with its set of 4-tuple-letter $W=\{W_1,\dots W_n\}$. Let $W\vert _{\v=v}=\{W_{(v,1)},W_{(v,2)},\dots W_{(v,k)}\}$ be a set of all $W_i \in W$ such that $\v(W_i) = v$. Then:
\begin{multline*}
    \mid \{(W_{(v,p)},W_{(v,q)}):W_{(v,p)}\text{ is descent top},\d(W_{(v,p)}) < \d(W_{(v,q)})\}\mid \\ 
    =\mid \{(W_{(v,p)},W_{(v,q)}):f(W_{(v,q)})\text{ is descent top},\d(W_{(v,p)}) < \d(W_{(v,q)})\}\mid 
\end{multline*}
\end{Lem}
\begin{proof}
We denote:
\begin{align*}
    M=\{(W_{(v,p)},W_{(v,q)}):W_{(v,p)}\text{ is descent top},\d(W_{(v,p)}) < \d(W_{(v,q)})\}\\
    N=\{(W_{(v,p)},W_{(v,q)}):f(W_{(v,q)})\text{ is descent top},\d(W_{(v,p)}) < \d(W_{(v,q)})\}
\end{align*}
The thing need to be proved equivalent to $\mid M\mid =\mid N\mid $.\par
According to $f_W$ construction, let $T_{v_e} = (W_{(v,1)},W_{(v,2)},\dots W_{(v,k)})$ be a tuple of all 4-tuple-letters in $W\vert _{\v=v}$ such that $\d(W_{(v,i)}) > \d(W_{(v,j)})$. Let $M^{\prime}=\{(W_{(v,p)},W_{(v,q)}):W_{(v,p)}\text{ is descent top}, W_{(v,p)}\in T_{v_e}, W_{(v,q)}\in T_{v_e}, p < q\}$. Similarly, let $N^{\prime}=\{(W_{(v,p)},W_{(v,q)}):f(W_{(v,q)})\text{ is descent top}, W_{(v,p)}\in T_{v_e}, W_{(v,q)}\in T_{v_e}, p < q\}$. It's clear that:
\begin{align*}
     M=M^{\prime} \text{ and } N=N^{\prime}
\end{align*}

We define a function $g:(W_{(v,p)},W_{(v,q)})\mapsto(W_{k+1-q},W_{k+1-p})$. We shall prove that with $x$ in $M^{\prime}$, the image of $x$ under $g$ is in $N^{\prime}$.\numbAtBack{a}

\noindent
Let $(W_{v_{p_0}},W_{v_{q_0}})$ be a tuple in $M^{\prime}$. Let $x_{v_{p_0}}$ and $x_{v_{q_0}}$ be a non-descent-top status of $W_{v_{p_0}}$ and $W_{v_{q_0}}$, respectively. Consider a pair $(W_{v_{k+1-q_0}},W_{v_{k+1-p_0}})$ in the set $M^{\prime}$. 
\begin{itemize}
    \item $W_{v_{k+1-q_0}}\in T_{v_e}$ and $W_{v_{k+1-p_0}}\in T_{v_e}$
    \item $W^{\prime}_{v_{k+1-p_0}}$ is the image of $W_{v_{k+1-p_0}}$ under $f_W$ and has non-descent-top status $x^{\prime}_{v_{k+1-p_0}} = x_{v_{p_0}}$. So that $f_W(W_{v_{k+1-p_0}})$ is a descent top if and only if $W_{v_{p_0}}$ is the descent top. Hence, $f_W(W_{v_{k+1-p_0}})$ is a descent top
    \item $k+1-q_0 < k+1-q_0$, since $p_0 < q_0$
\end{itemize}
Hence the tuple $(W_{v_{k+1-q_0}},W_{v_{k+1-p_0}})$ is in set $M^{\prime}$.

\noindent
Similarly, we can also prove that with $x \in N^{\prime}$, $g(x) \in M^{\prime}$.\numbAtBack{b}\\
(a) and (b) shows that:
\begin{align*}
    \mid N^{\prime}\mid  = \mid M^{\prime}\mid 
\end{align*}
This completes the proof of \cref{lma7}
\end{proof}

The following lemma can be proved similarly as \cref{lma7}
\begin{Lem}\label{lma8}
Given a word $w=w_1w_2\dots w_n$ in $\Sset_{(A,m)}$ with its set of 4-tuple-letter $W=\{W_1,\dots W_n\}$. Let $W\vert _{\v=v}=\{W_{(v,1)},W_{(v,2)},\dots W_{(v,k)}\}$ is the set of all $W_i \in W$ such that $\v(W_i) = v$. Then
\begin{multline*}
    \mid \{(W_{(v,p)},W_{(v,q)}):W_{(v,q)}\text{ is descent top},\d(W_{(v,p)}) < \d(W_{(v,q)})\}\mid \\ =\mid \{(W_{(v,p)},W_{(v,q)}):f(W_{(v,p)})\text{ is descent top},\d(W_{(v,p)}) < \d(W_{(v,q)})\}\mid 
\end{multline*}
\end{Lem}

\begin{Lem}\label{lma10}
Given a word $w=w_1w_2\dots w_n$ in $\Sset_{(A,m)}$ with its set of 4-tuple-letter $W=\{W_1,\dots W_n\}$. Let $U = \{U_1,U_2,\dots ,U_n\}$ in which $U_i$ is the image of $W_i$ under $\theta_W$. Let $w^{\prime} = W^{\prime}_1\dots W^{\prime}_n$ be a word constructed from $U$ under $\zeta$. Let $v$ be a letter in $w$. Then
\begin{align}
    \label{elma101}
    (\underline{21}2\vert _{2=v}) \text{ in } w = (2\underline{21}\vert _{2=v}) \text{ in } w^{\prime}\\
    \label{elma102}
    (2\underline{21}\vert _{2=v}) \text{ in } w = (\underline{21}2\vert _{2=v}) \text{ in } w^{\prime}
\end{align}
\end{Lem}
\begin{proof}
\begin{align}\label{elma103}
    \nonumber
    (\underline{21}2\vert _{2=v}) \text{ in } w 
    &= \mid \{(i,j):\v(w_i)=\v(w_j)=v , i<j , w_i\text{ is descent top}\}\mid \\
    &= \mid \{(W_i,W_j):\v(W_i)=\v(W_j)=v , i<j , W_i\text{ is descent top}\}\mid 
\end{align}
According to the way we define duplicate index in $w$:
\begin{alignat*}{1}
  \begin{cases} \v(W_i)=\v(W_j) \\ i < j \end{cases} &\iff \begin{cases} \v(W_i)=\v(W_j) \\ \d(W_i) < \d(W_j) \end{cases}
\end{alignat*}
Let $W\vert _{\v=v} = \{W_{(v,1)},W_{(v,2)},\dots\}$ be a set of all 4-tuple-letter that have value is $v$ in $W$. Assume $i=v_h$ and $j=v_k$
\begin{align}\label{elma104}
    \cref{elma103} = \mid \{(W_{(v,h)},W_{(v,k)}):\d(W_{(v,h)})<\d(W_{(v,k)}), W_{(v,h)}\text{ is descent top}\}\mid 
\end{align}
Applying the \cref{lma7}
\begin{align}
    \nonumber
    \cref{elma104} 
    &= \mid \{(W_{(v,h)},W_{(v,k)}):\d(W_{(v,h)})<\d(W_{(v,k)}), f(W_{(v,k)})\text{ is descent top}\}\mid \\
    \nonumber
    &= \mid \{(W_i,W_j) : \d(W_i) < \d(W_j) , \v(W_i)=\v(W_j)=v , f(W_j)\text{ is descent top}\}\mid \\
    \label{elma105}
    &= \mid \{(\theta_W(W_i),\theta_W(W_j)) : \d(W_i) < \d(W_j) , \v(W_i)=\v(W_j)=v , f(W_j)\text{ is descent top}\}\mid 
\end{align}
Since $\theta_W$ maintain the duplicate index and value of 4-tuple-letter
\begin{multline}\label{elma106}
    \cref{elma105}= \mid \{(U_i,U_j):\d(U_i)<\d(U_j), \v(U_i)=\v(U_j)=v, f(W_j)\text{ is descent top}\mid 
\end{multline}
Follow the \cref{lma2}, $f_W(W_j)$ is descent top if and only if $U_j$ is a descent top
\begin{align*}
    \cref{elma106} &= \mid \{(U_i,U_j):\d(U_i)<\d(U_j), \v(U_i)=\v(U_j)=v, U_j\text{ is descent top}\mid \\
    &= \mid \{({i^{\prime}}, {j^{\prime}}) : \v(W^{\prime}_{i^{\prime}}) = \v(W^{\prime}_{j^{\prime}}) = v , i^{\prime}<j^{\prime}, W^{\prime}_{j^{\prime}} \text{ is descent top}\}\mid \\
    &= (2\underline{21}\vert _{2=v}) \text{ in } w^{\prime}
\end{align*}
The equation \cref{elma101} has been proved.\\
The equation \cref{elma102} can be proved similarly, with the application of \cref{lma8}.
\end{proof}

\begin{Lem}\label{lma11}
Given a word $w=w_1w_2\dots w_n$ in $\Sset_{(A,m)}$ with its set of 4-tuple-letter $W=\{W_1,\dots W_n\}$. Let $U = \{U_1,U_2,\dots ,U_n\}$ in which $U_i$ is the image of $W_i$ under $\theta_W$. Let $w^{\prime} = W^{\prime}_1\dots W^{\prime}_n$ be a word constructed from $U$ under $\zeta$. Let $v$ be a letter in $w$. Then:
\begin{align}\label{elma111}
    (\underline{21}1\vert _{1=v}) \text{ in } w = (\underline{21}1\vert _{1=v}) \text{ in } w^{\prime}\\
    \label{elma112}
    (1\underline{21}\vert _{1=v}) \text{ in } w = (1\underline{21}\vert _{1=v}) \text{ in } w^{\prime}
\end{align}
\end{Lem}
\begin{proof}
The proof of this lemma is quite similar to the proof in \cref{lma10}, but it depends on the reservation of descent bottom status of $W_i$ under $\theta$.
\begin{align}\label{elma113}
    \nonumber
    (\underline{21}1\vert _{1=v}) \text{ in } w 
    &= \mid \{(i,j):\v(w_i)=\v(w_j)=v , i<j , w_i\text{ is descent bottom}\}\mid \\
    &= \mid \{(W_i,W_j):\v(W_i)=\v(W_j)=v , i<j , W_i\text{ is descent bottom}\}\mid 
\end{align}
According to the way we define duplicate index in $w$:
\begin{alignat*}{1}
  \begin{cases} \v(W_i)=\v(W_j) \\ i < j \end{cases} &\iff \begin{cases} \v(W_i)=\v(W_j) \\ \d(W_i) < \d(W_j) \end{cases}
\end{alignat*}
Let $W\vert _{\v=v} = \{W_{(v,1)},W_{(v,2)},\dots\}$ be a set of all 4-tuple-letter that have value is $v$ in $W$. Assume $i=v_h$ and $j=v_k$
\begin{align}\label{elma114}
\nonumber
    \cref{elma113} 
    &= \mid \{(W_{(v,h)},W_{(v,k)}):\d(W_{(v,h)})<\d(W_{(v,k)}), W_{(v,h)}\text{ is descent bottom}\}\mid \\
    &= \mid \{(\theta_W(W_i),\theta_W(W_j)) : \d(W_i) < \d(W_j) , \v(W_i)=\v(W_j)=v , W_i\text{ is descent bottom}\}\mid 
\end{align}
Since $\theta_W$ maintain the duplicate index and value of 4-tuple-letter
\begin{multline}\label{elma115}
    \cref{elma114}= \mid \{(U_i,U_j):\d(U_i)<\d(U_j), \v(U_i)=\v(U_j)=v, W_i\text{ is descent bottom}\mid 
\end{multline}
Follow the \cref{lma2}, $W_i$ is descent bottom if and only if $U_i$ is a descent bottom. Hence
\begin{align*}
    \cref{elma115} &= \mid \{(U_i,U_j):\d(U_i)<\d(U_j), \v(U_i)=\v(U_j)=v, U_i\text{ is descent bottom}\mid \\
    &= \mid \{({i^{\prime}}, {j^{\prime}}) : \v(W^{\prime}_{i^{\prime}}) = \v(W^{\prime}_{j^{\prime}}) = v , i^{\prime}<j^{\prime}, W^{\prime}_{i^{\prime}} \text{ is descent bottom}\}\mid \\
    &= (\underline{21}1\vert _{1=v}) \text{ in } w^{\prime}
\end{align*}
The equation \cref{elma111} has been proved. The equation \cref{elma112} can be proved similarly. 
\end{proof}

\begin{The}\label{the4}
\begin{multline*}
    ((2\underline{31}),(\underline{31}2),(\underline{21}1),(1\underline{21}),(\underline{21}2),(2\underline{21}),\underline{21})w\\
    = ((2\underline{31}),(\underline{31}2),(\underline{21}1),(1\underline{21}),(2\underline{21}),(\underline{21}2),\underline{21})\Phi(w)
\end{multline*}
\end{The}
\begin{proof}
We denote $w=w_1w_2\dots w_n$ and its set of 4-tuple-letter $W=\{W_1,\dots W_n\}$. Let $\Phi(w) = w^{\prime} = W^{\prime}_1\dots W^{\prime}_n$ and its 4-tuple-letter set $U$. Since $U$ is the image of $W$ under $\epsilon$, we can denote $U = \{U_1,U_2,\dots ,U_n\}$ in which $U_i$ is the image of $W_i$ under $\theta_W$.\\
Since we preserve the right embracing numbers through $\delta$, $\zeta$, and $\epsilon$,
\begin{align*}
    \r(U_i) = \r(W^{\prime}_j) = \r(W_i) = \l(w_i) \qquad \forall i \in \overline{1,n}
\end{align*}
Since each $W^{\prime}_j$ is the presence of only one $U_i$,
\begin{align*}
    \sum_{j=0}^{n}\r(W^{\prime}_j) &= \sum_{i=0}^{n}\l(w_i)\\
    \Res(w) &= \Res(\Phi(w))
\end{align*}
According to the definition \cref{def2} , the right embracing sum of the word in our consideration is counted by the number of occurrences of pattern 2\underline{31}.
\begin{align*}
    (2\underline{31})w = (2\underline{31})\Phi(w)
\end{align*}
Similarly, from \cref{lma6}, we conclude that:
\begin{align*}
    (\underline{31}2)w = (\underline{31}2)\Phi(w)
\end{align*}
Follow the \cref{lma3}, $\Destop(W) = \Destop(U)$ and $\Desbot(W) = \Desbot(U)$, hence the number of descent in $w$ equals the number of descent in $\Phi(w)$.
\begin{align*}
    (\underline{21})w = (\underline{21})\Phi(w)
\end{align*}

We shall compare $(\underline{21}1)$, $(1\underline{21})$, $(\underline{21}2)$, $(2\underline{21})$ of $w$ with $(\underline{21}1)$, $(1\underline{21})$, $2(\underline{21})$, $(\underline{21}2)$ of $\Phi(w)$, respectively.

We have
\begin{align}
    \label{the401}
    (\underline{21}1)w &= \sum_{v \in A}(\underline{21}1\vert _{1=v})w\\
    \label{the402}
    (\underline{21}1)\Phi(w) &= \sum_{v \in A}(\underline{21}1\vert _{1=v})\Phi(w)\\
    \label{the407}
    (\underline{21}1)w &= \sum_{v \in A}(\underline{21}1\vert _{1=v})w\\
    \label{the408}
    (\underline{21}1)\Phi(w) &= \sum_{v \in A}(\underline{21}1\vert _{1=v})\Phi(w)\\
    \label{the403}
    (2\underline{21})w &= \sum_{v \in A}(2_v\underline{21})w\\
    \label{the404}
    (\underline{21}2)\Phi(w) &= \sum_{v \in A}(\underline{21}2\vert _{2=v})\Phi(w)\\
    \label{the405}
    (\underline{21}2)w &= \sum_{v \in A}(\underline{21}2\vert _{2=v})w\\
    \label{the406}
    (2\underline{21})\Phi(w) &= \sum_{v \in A}(2\underline{21}\vert _{2=v})\Phi(w)
\end{align}
Follow the \cref{lma11}, for all $v$ in $A$, $(\underline{21}1\vert _{1=v})w = (\underline{21}1\vert _{1=v})\Phi(w)$. Applying this into \cref{the401} and \cref{the402}
\begin{align*}
    (\underline{21}1)w = (\underline{21}1)\Phi(w)
\end{align*}
Similarly, applying \cref{elma112} into \cref{the407} and \cref{the408};
\begin{align*}
    (1\underline{21})w = (1\underline{21})\Phi(w)
\end{align*}
Applying \cref{lma10} into \cref{the403}, \cref{the404}, \cref{the405} and \cref{the406} implies that
\begin{align*}
    (2\underline{21})w = (\underline{21}2)\Phi(w)\\
    (\underline{21}2)w = (2\underline{21})\Phi(w)
\end{align*}
\end{proof}
We notice that according to the \cref{lma3}, $\des w = \des \Phi(w)$. Therefore, the \cref{the5} and the \cref{the4} implies that:
\begin{Cor}
The involution $\Phi$ is an involution of $\Sset_{(A,m)}$ having the property:
\begin{align*}
    (\des, \mad_3)w &= (\des, \mad_1) \Phi(w)\\
    (\des, \mad_4)w &= (\des, \mad_2) \Phi(w)\\
    (\des, \madl_3)w &= (\des, \madl_1) \Phi(w)\\
    (\des, \madl_4)w &= (\des, \madl_2) \Phi(w)
\end{align*}
\end{Cor}
In \cite{CLARKE1997237}, $(\des, \mad_1)$, $(\des, \mad_2)$, $(\des, \madl_1)$, and $(\des, \madl_2)$ have been proved to be Eulerian-Mahonian. Therefore, all the pair $(\des, \mad_3)$, $(\des, \mad_4)$, $(\des, \madl_3)$, and $(\des, \madl_4)$ are also Eulerian-Mahonian.

\chapter{Conclusion}
Before summarizing the results of this thesis, we would like to emphasize how the idea about a bijection $\Phi$ has been made. The function $\Phi$ was constructed after defining $\mad$ and $\madl$ extensions. The first statistic we wanted to transform is $\mad_3$. After trying several possible approaches, we decided to create a bijection that transforms mad-3 to mad-1. Since these two statistics extensions have many similarities in their vincular patterns representations, we naturally want to construct a function transforming patterns into patterns. The function should maintain the patterns $2\underline{31}$, $\underline{31}2$, $\underline{21}$, $\underline{21}1$ and transform $\underline{21}2$ to $2\underline{21}$. We consider constructing $\Phi$ as an involution to support bijective proof. To achieve this, we require $\Phi$ also need to transform $2\underline{21}$ to $\underline{21}2$. We notice that $\underline{21}$ is the pattern presented for a decent block, in which the number $2$ presents for descent top. Naturally, we came up with the idea of swapping the descent top status of letters that have the same values but have opposite relative positions. Relative positions are places of all letter that has the same value, for instance, relative position of $4$-letters in word $4_124_2354_31$ is $4_1\ 4_2\ 4_3$. $4_1$'s place is opposite to $4_3$'s place.

\noindent
The skeleton concept in \cite{CLARKE1997237} gives us the idea for maintaining patterns $\underline{31}2$ and $2\underline{31}$ is to split the word into letters attached with their attributes (function $\delta$), transforms these letters into letters with expected attributes (function $\epsilon$) and then rearrange these into new word that represent correctly expected attributes (function $\zeta$). But since the skeleton algorithm in \cite{CLARKE1997237} is for permutations only, we tailored this algorithm into $n_i$ insertion algorithm. The $n_i$ insertion algorithm ensures  not only $\underline{31}2$ and $2\underline{31}$ statistics of letters but also ensures the desired relative order between letters, which affects to the counting of patterns $\underline{21}2$, $2\underline{21}$, $\underline{21}1$, $1\underline{21}$, are correct.

To sum up, in this thesis, we have defined all possible $\mad$, $\madl$ extensions that are Mahonian on repetitive permutations, which are also Eulerian-Mahonian statistics. Some of them were defined and proved to be Eulerian-Mahonian \cite{CLARKE1997237}, while the rest are new Eulerian-Mahonian statistics extensions on repetitive permutations. With the statistics extensions in \cite{CLARKE1997237}, we represent them as vincular patterns combinations. With the new statistics extensions, we construct with proofs an involution $\Phi$ that transform these new statistics extensions to the statistics extensions that have been proved to be Eulerian-Mahonian in \cite{CLARKE1997237}. The function $\Phi$ also preserves statistics $\des$, proving that these new statistics extensions are Eulerian-Mahonian.

\paragraph{Future works} We have tried to find all possible Mahonian statistics extensions on repetitive permutations not only for $\mad$, $\madl$ but also for other statistics. We have found a vincular patterns extension for $\stat^{\prime}$ promising to be Mahonian on repetitive permutations $$\stat^{\prime} = \underline{31}2 + \underline{13}2 + \underline{32}1 + \underline{21} + \underline{22}1 + \underline{21}1$$This vincular patterns combination is our next extension to pursue in the future.
\printbibliography[
heading=bibintoc,
title={References}
]

\end{document}